\documentclass[11pt]{article}

\usepackage[font=small,labelfont=bf]{caption}
\usepackage[a4paper, total={8.5in, 11in},margin=1in]{geometry}

\usepackage[utf8]{inputenc}
\usepackage{amsmath}
\usepackage{amssymb}
\usepackage{amsthm}
\usepackage{mathtools}
\usepackage{thmtools}
\usepackage{placeins}
\usepackage{array}
\usepackage{float}

\usepackage{hyperref}
\hypersetup{colorlinks=true,linkcolor=red,citecolor=blue}

\usepackage{nameref,cleveref}
\usepackage{stackengine}
\usepackage[ruled,vlined]{algorithm2e}
\newtheorem{theorem}{Theorem}[section]
\newtheorem{definition}{Definition}[section]
\newtheorem{proposition}{Proposition}[section]
\newtheorem{corollary}{Corollary}[section]
\newtheorem{lemma}{Lemma}[section]

\newenvironment{sproof}{%
  \proof}{\endproof}

\DeclareMathOperator*{\NE}{NE}

\DeclareMathOperator*{\argmax}{arg\,max}
\DeclareMathOperator*{\argmin}{arg\,min}

\makeatletter
\def\env@dmatrix{\hskip -\arraycolsep
  \let\@ifnextchar\new@ifnextchar
  \extrarowheight=2ex
  \array{*\c@MaxMatrixCols{>{\displaystyle}c}}}

\newenvironment{bdmatrix}
  {\left(\env@dmatrix}
  {\endmatrix\right)}
\makeatother

\author{
  Ioannis Anagnostides\\
  \texttt{ianagnost@student.ethz.ch}
  \and
  Paolo Penna\\
  \texttt{paolo.penna@inf.ethz.ch}
}
\date{}
\title{Solving Zero-Sum Games through Alternating Projections}

\begin{document}

\setcounter{page}{0}

\maketitle

\begin{abstract}

In this work, we establish near-linear and strong convergence for a natural first-order iterative algorithm that simulates Von Neumann's Alternating Projections method in zero-sum games. First, we provide a precise analysis of Optimistic Gradient Descent/Ascent (OGDA) -- an optimistic variant of Gradient Descent/Ascent -- for \emph{unconstrained} bilinear games, extending and strengthening prior results along several directions. Our characterization is based on a closed-form solution we derive for the dynamics, while our results also reveal several surprising properties. Indeed, our main algorithmic contribution is founded on a geometric feature of OGDA we discovered; namely, the limit points of the dynamics are the orthogonal projection of the initial state to the space of attractors.

Motivated by this property, we show that the equilibria for a natural class of \emph{constrained} bilinear games are the intersection of the unconstrained stationary points with the corresponding probability simplexes. Thus, we employ OGDA to implement an Alternating Projections procedure, converging to an $\epsilon$-approximate Nash equilibrium in $\widetilde{\mathcal{O}}(\log^2(1/\epsilon))$ iterations. Our techniques supplement the recent work in pursuing last-iterate guarantees in min-max optimization. Finally,  we illustrate an -- in principle -- trivial reduction from any game to the assumed class of instances, without altering the space of equilibria.
\end{abstract}

\clearpage

\section{Introduction}

The classical problem of finding a Nash equilibrium in multi-agent systems has been a topic of prolific research in several areas, including Mathematics, Economics, Algorithmic Game Theory, Optimization \cite{1950PNAS...36...48N,sion1958,10.1145/1132516.1132527,10.5555/1296179,10.1007/s10107-004-0552-5} and more recently Machine Learning in the context of Generative Adversarial Networks \cite{NIPS2014_5423,arjovsky2017wasserstein} and multi-agent reinforcement learning \cite{10.5555/645527.657296}. The inception of this endeavor can be traced back to Von Neumann's celebrated min-max theorem, which asserts that

\begin{equation}
    \label{equation:min-max}
    \min_{\mathbf{x} \in \Delta_n} \max_{\mathbf{y} \in \Delta_m} \mathbf{x}^T A \mathbf{y} = \max_{\mathbf{y} \in \Delta_m} \min_{\mathbf{x} \in \Delta_n}  \mathbf{x}^T A \mathbf{y},
\end{equation}
where $\Delta_n$ and $\Delta_m$ are probability simplexes and $A$ the matrix of the game. To be more precise, \eqref{equation:min-max} implies that an equilibrium -- a pair of randomized strategies such that neither player can benefit from a unilateral deviation -- always exists; yet, the min-max theorem does not inform us on whether natural learning algorithms can converge to this equilibrium with a reasonable amount of computational resources. This question has given rise to intense research, commencing from the analysis of fictitious play by J. Robinson \cite{10.2307/1969530} and leading to the development of the no-regret framework \cite{cesa-bianchi_lugosi_2006,10.1007/11503415_42,v008a006}. However, an unsatisfactory characteristic of these results is that the minimax pair is only obtained in an average sense, without any last-iterate guarantees. Indeed, a regret-based analysis cannot distinguish between a self-stabilizing system and one with recurrent cycles. In this context, it has been extensively documented that limit cycles -- or more precisely \emph{Poincaré recurrences} -- are persistent in broad families of no-regret schemes, such as Mirror Descent and Follow-The-Regularized-Leader \cite{mertikopoulos2017cycles, palaiopanos2017multiplicative, 10.1145/2840728.2840757, doi:10.1137/1.9781611973402.64}.

\begin{figure}[!ht]
    \centering
    \includegraphics[scale=0.55]{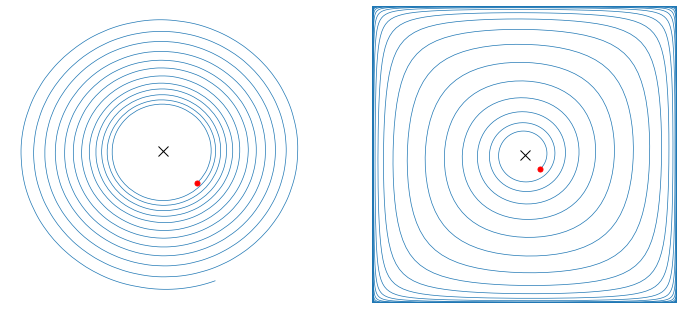}
    \caption{The behavior of Projected Gradient Descent (leftmost image) and Entropic - or Exponentiated - Descent (rightmost image) for the game of matching pennies. Symbol 'x' represents the unique (uniform) equilibrium, while the red point corresponds to the initial state of the system. Although both algorithms suffer no-regret -- as instances of Mirror Descent, they exhibit cyclic behavior around the equilibrium.}
\end{figure}

The crucial issue of last-iterate convergence was addressed by Daskalakis et al.~\cite{DBLP:journals/corr/abs-1711-00141} in two-player and zero-sum games in the context of training Generative Adversarial Networks. In particular, they introduced Optimistic Gradient Descent/Ascent (henceforth abbreviated as OGDA), a simple variant of Gradient Descent/Ascent (GDA) that incorporates a prediction term on the next iteration's gradient, and they proved point-wise convergence in unconstrained  bilinear min-max optimization problems. It is clear that the stabilizing effect of \emph{optimism} is of fundamental importance within the scope of Game Theory, which endeavors to study and indeed, control dynamical systems of autonomous agents. However, despite recent follow-up works \cite{liang2018interaction,mokhtari2019unified}, an exact characterization of OGDA and its properties remains largely unresolved. In this context, the first part of our work provides a precise analysis of OGDA, resolving several open issues and illuminating surprising attributes of the dynamics.

Naturally, the stability of the learning algorithm has also emerged as a critical consideration in the more challenging and indeed, relevant case of constrained zero-sum -- or simply zero-sum -- games. Specifically, Daskalakis and Panageas \cite{daskalakis2018lastiterate} analyzed an optimistic variant of Multiplicative Weight Updates and they showed last-iterate convergence in games with a unique equilibrium; yet, the rate of convergence was left entirely open. The same issue arises in the analysis of Optimistic Mirror Descent by Mertikopoulos et al.~\cite{mertikopoulos2018optimistic}; in the parlance of the Variational Inequality framework, the main challenge is that the operator -- the flow -- in bilinear games is weakly-monotone -- just as linear functions are convex in a weak sense -- and standard tools appear to be of no use. The main algorithmic contribution of our paper is a natural variant of projected OGDA that simulates an Alternating Projections procedure; as such, it exhibits a surprising near-linear and strong convergence in constrained games. Our approach is based on a non-trivial geometric property of OGDA: the limit points are -- under certain hypotheses -- positively correlated with the Nash equilibria of the constrained game. Although our algorithm is not admitted in the no-regret framework, it is based entirely on the no-regret (projected) OGDA dynamics. Importantly, our techniques and the connections we establish could prove useful in the characterization of the convergence's rate in other algorithms.

\paragraph{Related Work}
%\label{section:related_work}

Our work follows the line of research initiated by \cite{DBLP:journals/corr/abs-1711-00141}; one of their key results was proving through an inductive argument that OGDA, a simple variant of GDA, exhibits point-wise convergence to the space of Nash equilibria for any unconstrained  bilinear game. The convergence rate was later shown to be linear in \cite{liang2018interaction}. Moreover, this result was also proven in \cite{mokhtari2019unified}, where the authors illustrated that OGDA can be examined as an approximate version of the proximal point method; through this prism, they also provided a convergence guarantee for convex-concave games. However, both of the aforementioned works consider only the special case of a square  matrix with full-rank.

The underlying technique of optimism employed in OGDA has gradually emerged in the fields of Online Learning and Convex Optimization \cite{NIPS2015_5763,pmlr-v32-steinhardtb14,wang2018acceleration}, and has led to very natural algorithmic paradigms. More precisely, optimism consists of exploiting the predictability or smoothness of the future cost functions, incorporating additional knowledge in the optimization step. In fact, in the context of zero-sum games, this technique has guided the optimal no-regret algorithms, leading to a regret of $\Theta(1/T)$, where $T$ the number of iterations \cite{DASKALAKIS2015327, NIPS2013_5147}.

It is well-established that a minimax pair in zero-sum games can be computed in a centralized manner via an induced Linear Program (LP). In particular, interior point methods and variants of Newton's method are known to exhibit very fast convergence, even super-linear -- e.g. $\mathcal{O}(\log \log \left( 1/\epsilon \right))$; see \cite{POTRA2000281,10.5555/265009}. We also refer to \cite{Rebennack2009} for the ellipsoid method, another celebrated approach for solving LPs. However, these algorithms are not commonly used in practice for a number of reasons. Indeed, first-order methods have dominated the attention both in the literature and in practical applications, mainly due to their simplicity, their robustness - e.g. noise tolerance - and the limited computational cost of implementing a single step, relatively to the aforementioned algorithms \cite{ruder2016overview}. Another reference point for our work should be Nesterov's iterative algorithm \cite{10.1007/s10107-004-0552-5} that strongly converges to an $\epsilon$-approximate equilibrium after $\mathcal{O}(1/\epsilon)$ iterations.

In a closely related work, Mertikopoulos et. al \cite{mertikopoulos2018optimistic} established convergence for projected OGDA and Extra-Gradient dynamics. Our main algorithm is also established as projected variant of OGDA, but instead of projecting after a single step, the projection is performed after multiple iterations. A no-regret analysis of Optimistic Mirror Descent can be found in \cite{kangarshahi2018lets}. For a characterization of the stationary points of GDA and OGDA and their nexus to Nash equilibria beyond convex-concave settings we refer to \cite{daskalakis2018limit}. Finally, several algorithms have been proposed specifically for solving saddle-point problems; we refer to \cite{adolphs2018local, schfer2019competitive, mazumdar2019finding} and references thereof.

\paragraph{Our Contributions}
In the first part of our work (\Cref{section:OGDA}), we provide a thorough and comprehensive analysis of Optimistic Gradient Descent/Ascent in unconstrained bilinear games. Specifically, we first derive an exact and concise solution to the dynamics in Subsection \ref{subsection:dynamics-solution}; through this result, we establish convergence from any initial state (\Cref{theorem:convergence}). Note that the guarantee in \cite{DBLP:journals/corr/abs-1711-00141} was only shown for cherry-picked initial conditions, an unnatural restriction for a linear dynamical system. We also show that the rate of convergence is linear (\Cref{corollary:convergence}); this was proved in follow-up works \cite{liang2018interaction,mokhtari2019unified} for the specific case of a square  non-singular matrix. Our analysis is also tighter with respect to the learning rate, yielding a broader stability region compared to the aforementioned works. In fact, in \Cref{proposition:convergence_rate} we derive the exact rate of convergence with respect to the learning rate and the spectrum of the matrix. One important implication is that within our stable region, increasing the learning rate will accelerate the convergence of the system, a property which is of clear significance in practical implementations of the algorithm. \Cref{proposition:convergence_rate} also implies a surprising discontinuity on the behavior of the system: an arbitrarily small noise may dramatically alter the rate of convergence (see \Cref{section:discontinuity}). Another counter-intuitive consequence of our results is that OGDA could actually converge with negative learning rate. Finally, we reformulate our solution to illustrate the inherent oscillatory components in the dynamics, albeit with a vanishing amplitude (\Cref{section:oscillations}). Throughout \Cref{section:OGDA}, we mainly use techniques and tools from Linear Algebra and Spectral Analysis -- such as the Spectral Decomposition, while we also derive several non-trivial identities that could of independent interest. We consider our analysis to be simpler than the existing ones.

In the second part, we turn to constrained games and we commence (in \Cref{section:nash-constrained}) by showing that the Nash equilibria are the intersection of the unconstrained stationary points with the corresponding probability simplexes, assuming that the value of the (constrained) game is zero and that there exists an interior equilibrium (see \Cref{proposition:NE}). Thus, under these conditions, the optimal strategies can be expressed as the intersection of two closed and convex sets and hence, they can be determined through Von Neumann's Alternating Projections method. In this context, our algorithm (Algorithm \ref{algorithm:projected-OGDA}) is based on a non-trivial structural property of OGDA: the limit points are the orthogonal projection of the initial state to the space of attractors (see \Cref{theorem:projection} for a formal statement). As a result, OGDA is employed in order to implement the Alternating Projections procedure, and subsequently, we manage to show with simple arguments that our algorithm yields an $\epsilon$-approximate Nash equilibrium with $\widetilde{\mathcal{O}}(\log^2 (1/\epsilon))$ iterations (\Cref{theorem:fast_convergence}).

Moreover, our algorithm also converges in the last-iterate sense, supplementing the recent works in this particular direction \cite{mertikopoulos2018optimistic,daskalakis2018lastiterate}; however, unlike the aforementioned results, our analysis gives a precise and parametric characterization of the convergence's rate; thus, our techniques could prove useful in the analysis of other dynamics. We also believe that the variant of projected Gradient Descent we introduced, namely performing the projection to the feasible set only after multiple iterations, could be of independent interest in the regime of Optimization. The main caveat of our approach is that we consider (constrained) games that have an interior equilibrium and value $v = 0$; we ameliorate this limitation by showing an -- in principle -- trivial reduction from any arbitrary game to the assumed class of instances. We emphasize on the fact that the equilibria of the game remain -- essentially -- invariant under our reduction. Importantly, our approach offers a novel algorithmic viewpoint on a fundamental problem in Optimization and Game Theory.

\paragraph{Alternating Projections} Our main algorithm (Algorithm \ref{algorithm:projected-OGDA}) is an approximate instance of Von Neumann's celebrated method of Alternating Projections. More precisely, if $\mathcal{H}_1$ and $\mathcal{H}_2$ are closed subspaces of a Hilbert space $\mathcal{H}$, he showed that alternatively projecting to $\mathcal{H}_1$ and $\mathcal{H}_2$ will eventually converge in norm to the projection of the initial point to $\mathcal{H}_1 \cap \mathcal{H}_2$. This method has been applied in many areas, including Stochastic Processes \cite{10.1137/1113049}, solving linear equations \cite{10.1007/BF01436376} and Convex Optimization \cite{10.2307/2132495,BREGMAN1967200,10.5555/993483}, while it was also independently discovered by Wiener \cite{Wiener1955}. Moreover, a generalization of this algorithm to multiple closed subspaces was given by Halperin \cite{halperin}, introducing the method of \emph{cyclic alternating projections}. The rate of convergence was addressed in \cite{10.2307/1990404,badea2010rate,10.1007/BF02551235}, and in particular it was shown that for $2$ closed subspaces -- and under mild assumptions -- the rate of convergence is linear and parameterized by the \emph{Friedrichs angle} between the subspaces. Beyond affine spaces, Von Neumann's method also converges when the sets are closed and convex \footnote{In this context, the method of alternating projections is sometimes referred to as projection onto convex sets or POCS}, assuming a non-empty intersection. In fact, certain results in this domain are applicable even when the intersection between the subsets is empty; e.g. see Dijkstra's projection algorithm.

\section{Preliminaries}
\label{section:preliminaries}

\paragraph{Nash equilibrium} Consider a continuously differentiable function $f: \mathcal{X} \times \mathcal{Y} \mapsto \mathbb{R}$ that represents the objective function of the game, with $f(\mathbf{x}, \mathbf{y})$ the payoff of player $\mathcal{X}$ to player $\mathcal{Y}$ under strategies $\mathbf{x} \in \mathcal{X}$ and $\mathbf{y} \in \mathcal{Y}$ respectively. A pair of strategies $(\mathbf{x}^*, \mathbf{y}^*) \in \mathcal{X} \times \mathcal{Y}$ is a Nash equilibrium -- or a \emph{saddle-point} of $f(\mathbf{x}, \mathbf{y})$ -- if $\forall (\mathbf{x}, \mathbf{y}) \in \mathcal{X} \times \mathcal{Y}$,

\begin{equation}
    \label{equation:definition-NE}
 f(\mathbf{x}^*, \mathbf{y}) \leq f(\mathbf{x}^*, \mathbf{y}^*) \leq f(\mathbf{x}, \mathbf{y}^*).
\end{equation}

A pair of strategies $(\mathbf{x}^*, \mathbf{y}^*)$ will be referred to as an $\epsilon$-\emph{approximate} Nash equilibrium if it satisfies \eqref{equation:definition-NE} with up to an $\epsilon$ additive error. We may sometimes refer to player $\mathcal{X}$ as the minimizer and player $\mathcal{Y}$ as the maximizer. Throughout this paper, we consider exclusively the case of $f(\mathbf{x}, \mathbf{y}) = \mathbf{x}^T A \mathbf{y}$, while in a constrained game -- or simply (finite) zero-sum game in the literature of Game Theory -- we additionally require that $\mathcal{X} = \Delta_n$ and $\mathcal{Y} = \Delta_m$.

\paragraph{Optimistic Gradient Descent/Ascent} Let us focus on the unconstrained case, i.e. $\mathcal{X} = \mathbb{R}^n$ and $\mathcal{Y} = \mathbb{R}^m$. The most natural optimization algorithm for solving the induced saddle-point problem is by performing simultaneously Gradient Descent on $\mathbf{x}$ and Gradient Ascent on $\mathbf{y}$; formally, if $\eta>0$ denotes some positive constant -- typically referred to as the \emph{learning rate}, GDA can be described as follows:
\begin{equation}
    \label{equation:GDA}
    \begin{split}
    \mathbf{x}_{t} = \mathbf{x}_{t-1} - \eta \nabla_{\mathbf{x}} f(\mathbf{x}_{t-1}, \mathbf{y}_{t-1}), \\
    \mathbf{y}_{t} = \mathbf{y}_{t-1} + \eta \nabla_{\mathbf{y}} f(\mathbf{x}_{t-1}, \mathbf{y}_{t-1}).
    \end{split}
\end{equation}

However, there are very simple examples where the system of equations \eqref{equation:GDA} diverges; for instance, when $f(x, y) = x y$ with $x, y \in \mathbb{R}$ and $(x_0, y_0) \neq (0, 0)$, GDA is known to diverge for any learning rate $\eta > 0$. This inadequacy has motivated \emph{optimistic} variants of GDA that incorporate some prediction on the next iteration's gradient through the regularization term (recall that Gradient Descent can be viewed as an instance of Follow-The-Regularized-Leader (FTRL) with Euclidean regularizer \cite{MAL-018}). With OGDA we refer to the optimization variant that arises when the prediction of the next iteration's gradient is simply the previously observed gradient; this yields the following update rules:
\begin{equation}
    \label{equation:OGDA}
    \begin{split}
    \mathbf{x}_{t} = \mathbf{x}_{t-1} - 2 \eta \nabla_{\mathbf{x}} f(\mathbf{x}_{t-1}, \mathbf{y}_{t-1}) + \eta \nabla_{\mathbf{x}} f(\mathbf{x}_{t-2}, \mathbf{y}_{t-2}), \\
    \mathbf{y}_{t} = \mathbf{y}_{t-1} + 2 \eta \nabla_{\mathbf{y}} f(\mathbf{x}_{t-1}, \mathbf{y}_{t-1}) - \eta \nabla_{\mathbf{y}} f(\mathbf{x}_{t-2}, \mathbf{y}_{t-2}).
    \end{split}
\end{equation}

A probability vector $\mathbf{p}$ is referred to as \emph{interior} if $\mathbf{p}_i > 0, \forall i$; subsequently, an equilibrium point $(\mathbf{x}, \mathbf{y}) \in \Delta_n \times \Delta_m$ is called interior if both $\mathbf{x}$ and $\mathbf{y}$ are interior probability vectors. Let $a_{max}$ the entry of $A$ with maximum absolute value; it follows that the bilinear objective function $f(\mathbf{x}, \mathbf{y}) = \mathbf{x}^T A \mathbf{y} $ is $|a_{max}|$-Lipschitz continuous in $\Delta_n \times \Delta_m$. Note that rescaling the entries of the matrix by a positive constant does not alter the game; for \Cref{theorem:fast_convergence} we make the normative assumption that every entry is multiplied by $1/|a_{max}|$, so that the objective function is $1$-Lipschitz. A set is \emph{closed} if and only if it contains all of its limit points. Moreover, a subset of a Euclidean space is \emph{compact} if and only if it is closed and bounded (Heine-Borel theorem).

\paragraph{Hilbert Spaces} Although this paper deals with optimization in finite-dimensional Euclidean spaces, we wish to present the Alternating Projections method in its general form in arbitrary Hilbert spaces (see \Cref{theorem:alternating_projections}). In this context, we review some basic concepts that the reader may need. A Hilbert space $\mathcal{H}$ is an inner product and complete -- every Cauchy sequence converges -- metric space. The norm is defined as $||u|| = \sqrt{\langle u, u \rangle}$, where $\langle \cdot, \cdot \rangle$ denotes the inner product in $\mathcal{H}$. Moreover, the distance between $u, v \in \mathcal{H}$ is defined in terms of norm by $d(u, v) = ||u - v||$, while $d(u, S) = \inf_{v \in S} d(u, v)$ -- for some non-empty set $S \subseteq \mathcal{H}$. Also recall from Hilbert's projection theorem that if $C$ is a non-empty, closed and convex set in a Hilbert space $\mathcal{H}$, the projection of any point $u \in \mathcal{H}$ to $C$ is uniquely defined as $\mathcal{P}_C(u) = \argmin_{v \in C} || u - v||$.

\begin{definition}
A sequence $\left\{ u_n \right\}_0^{\infty}$ on a Hilbert space $\mathcal{H}$ is said to converge linearly \footnote{To avoid confusion, we should mention that linear convergence in iterative methods is usually referred to as exponential convergence in discretization methods.} to $u^* \in \mathcal{H}$ with rate $\lambda \in (0, 1)$ if
\begin{equation}
    \lim_{n \to \infty} \frac{|| u_{n+1} - u^* ||}{|| u_n - u^* ||} = \lambda.
\end{equation}
\end{definition}

\paragraph{Notation} We denote with $A \in \mathbb{R}^{n \times m}$ the matrix of the game and with $\mathbf{x} \in \mathcal{X}, \mathbf{y} \in \mathcal{Y}$ the players' strategies. With $\Delta_k$ we denote the $k$-dimensional probability simplex. We use $t$ and $k$ as discrete time indexes, while the variables $i, j, k, t$ will be implied to be integers, without explicitly stated as such. We use $\mathbf{I}_{k}$ and $\mathbf{0}_{k \times \ell}$ to refer to the identity matrix of size $k \times k$ and the zero matrix of size $k \times \ell$ respectively; when $k = \ell$ we simply write $\mathbf{0}_{k}$ instead of $\mathbf{0}_{k \times k}$. We also use $\mathbf{1}_{k \times \ell}$ to denote the $k \times \ell$ matrix with $1$ in every entry, while we sometimes omit the dimensions when they are clear from the context. The vector norm will refer to the Euclidean norm $||\cdot||_2$, while we denote with $|| S ||$ the spectral norm of $S$, that is, the square root of the maximum eigenvalue of $S^T S$. Finally, $\mathcal{N}(S)$ represents the null space of matrix $S$.

\section{Characterizing OGDA}
\label{section:OGDA}

Throughout this section, we analyze Optimistic Gradient Descent/Ascent for unconstrained and bilinear games, i.e. $f(\mathbf{x}, \mathbf{y}) = \mathbf{x}^T A \mathbf{y}$, $\mathcal{X} = \mathbb{R}^n$ and $\mathcal{Y} = \mathbb{R}^m$.

\subsection{Solving the Dynamics}
\label{subsection:dynamics-solution}

In this subsection, we derive a closed-form and concise solution to the OGDA dynamics. First, consider some arbitrary initial conditions for the players' strategies $\mathbf{x}_{-1}, \mathbf{x}_0 \in \mathbb{R}^n$ and $\mathbf{y}_{-1}, \mathbf{y}_0 \in \mathbb{R}^m$. When the objective function is bilinear, the update rules of OGDA \eqref{equation:OGDA} can be formulated for $t \geq 1$ as
\begin{equation}
    \label{equation:OGDA-bilinear}
    \begin{split}
    \mathbf{x}_{t} &= \mathbf{x}_{t-1} - 2 \eta A \mathbf{y}_{t-1} + \eta A \mathbf{y}_{t-2}, \\
    \mathbf{y}_{t} &= \mathbf{y}_{t-1} + 2 \eta A^T \mathbf{x}_{t-1} - \eta A^T \mathbf{x}_{t-2}.
    \end{split}
\end{equation}
These equations can be expressed more concisely in matrix form:

\begin{equation}
    \label{equation:OGDA-matrix_form}
    \begin{pmatrix}
    \mathbf{x}_{t} \\
    \mathbf{y}_{t}
    \end{pmatrix}
    =
    \begin{pmatrix}
    \mathbf{I}_n & -2 \eta A \\
    2 \eta A^T & \mathbf{I}_m
    \end{pmatrix}
    \begin{pmatrix}
    \mathbf{x}_{t-1} \\
    \mathbf{y}_{t-1}
    \end{pmatrix}
    +
    \begin{pmatrix}
    \mathbf{0}_n & \eta A \\
    -\eta A^T & \mathbf{0}_m
    \end{pmatrix}
    \begin{pmatrix}
    \mathbf{x}_{t-2} \\
    \mathbf{y}_{t-2}
    \end{pmatrix}.
\end{equation}
\bigbreak
\noindent
In correspondence to the last expression, let us introduce the following matrices:

\begin{equation}
    \mathbf{z}_t =
    \begin{pmatrix}
    \mathbf{x}_t \\
    \mathbf{y}_t
    \end{pmatrix},
    B =
    \begin{pmatrix}
    \mathbf{I}_n & -2 \eta A \\
    2\eta A^T & \mathbf{I}_m
    \end{pmatrix},
    C =
    \begin{pmatrix}
    \mathbf{0}_{n} & \eta A \\
    -\eta A^T & \mathbf{0}_m
    \end{pmatrix}.
\end{equation}
With this notation, \Cref{equation:OGDA-matrix_form} can be re-written as

\begin{equation}
    \label{equation:recursion-second_order}
    \mathbf{z}_t = B \mathbf{z}_{t-1} + C \mathbf{z}_{t-2}.
\end{equation}

\Cref{equation:recursion-second_order} induces a second-order and linear recursion in matrix form; hence, its solution can be derived through a standard technique. In particular, it is easy to verify that \eqref{equation:recursion-second_order} can be equivalently reformulated as

\begin{equation}
    \label{equation:recursion-first_order}
    \begin{pmatrix}
    \mathbf{z}_{t} \\
    \mathbf{z}_{t-1}
    \end{pmatrix}
    =
    \begin{pmatrix}
    B & C \\
    \mathbf{I}_{n + m} & \mathbf{0}_{n+m}
    \end{pmatrix}
    \begin{pmatrix}
    \mathbf{z}_{t-1} \\
    \mathbf{z}_{t-2}
    \end{pmatrix}
    = \Delta^t
    \begin{pmatrix}
    \mathbf{z}_{0} \\
    \mathbf{z}_{-1}
    \end{pmatrix},
\end{equation}
where

\begin{equation}
    \Delta
    =
    \begin{pmatrix}
    B & C \\
    \mathbf{I}_{n + m} & \mathbf{0}_{n+m}
    \end{pmatrix}.
\end{equation}
\bigbreak

As a result, we have reduced solving the OGDA dynamics to determining matrix $\Delta^t$. To this end, first note that the block sub-matrices of $\Delta$ have the same dimensions $(n+m) \times (n+m)$. In addition, every couple of these sub-matrices satisfies the commutative property; indeed, we can verify that $B C = C B$ and hence, every possible multiplicative combination will also commute. Thus, we can invoke for the sub-matrices polynomial identities that apply for scalar numbers. In this direction, we establish the following claim:

\begin{lemma}
\label{lemma:matrix_power}
Consider a matrix $R \in \mathbb{R}^{2 \times 2}$ defined as

\begin{equation}
    \label{equation:matrix_R}
   R
    =
    \begin{pmatrix}
    b & c \\
    1 & 0
    \end{pmatrix}.
\end{equation}

\bigbreak
\noindent
Then, for any $k \geq 2$,

\begin{equation}
    \label{equation:matrix_power}
    R^k
    =
    \begin{bdmatrix}
    \sum_{i=0}^{\lfloor \frac{k}{2} \rfloor} \binom{k - i}{i} b^{k - 2i} c^i & \sum_{i=0}^{\lfloor \frac{k-1}{2} \rfloor} \binom{k - i - 1}{i} b^{k - 2i - 1} c^{i+1} \\
    \sum_{i=0}^{\lfloor \frac{k-1}{2} \rfloor} \binom{k - i - 1}{i} b^{k - 2i - 1} c^i & \sum_{i=0}^{\lfloor\frac{k-2}{2} \rfloor} \binom{k - i - 2}{i} b^{k - 2i - 2} c^{i+1}
    \end{bdmatrix}.
\end{equation}

\end{lemma}
\bigbreak
We refer to \Cref{subsection:proof_1} for a proof of this lemma; essentially it follows from the Cayley–Hamilton theorem and simple combinatorial arguments. In addition, we can also express the power of matrix $R$  -- as defined in \Cref{equation:matrix_R} -- using the spectral decomposition -- or the more general Singular Value Decomposition (SVD) method, which yields the following identity:

\begin{proposition}

    For any $b,c \in \mathbb{R}$ such that $b^2 + 4c > 0$ and $k \geq 0$,
    \label{proposition:binomial_variant}
    \begin{equation}
        \sum_{i=0}^{\lfloor \frac{k}{2} \rfloor} \binom{k-i}{i} b^{k-2i}c^i = \frac{1}{\sqrt{b^2 + 4c}} \left\{ \left( \frac{b + \sqrt{b^2 + 4c}}{2} \right)^{k+1} - \left( \frac{b - \sqrt{b^2 + 4c}}{2} \right)^{k+1} \right\}.
    \end{equation}

\end{proposition}

\bigbreak
This identity can be seen as a non-trivial variant of the celebrated binomial theorem. For a proof and connections to the Fibonacci series we refer to \Cref{subsection:proof_2}. Next, we employ the previous results to derive analogous expressions for the matrix case.

\begin{corollary}
\label{corollary:Delta_power}
For any $k \geq 2$,

\begin{equation}
    \label{equation:Delta_power}
    \Delta^k
    =
    \begin{bdmatrix}
    \sum_{i=0}^{\lfloor \frac{k}{2} \rfloor} \binom{k - i}{i} B^{k - 2i} C^i & \sum_{i=0}^{\lfloor \frac{k-1}{2} \rfloor} \binom{k - i - 1}{i} B^{k - 2i - 1} C^{i+1} \\
    \sum_{i=0}^{\lfloor \frac{k-1}{2} \rfloor} \binom{k - i - 1}{i} B^{k - 2i - 1} C^i & \sum_{i=0}^{\lfloor \frac{k-2}{2} \rfloor} \binom{k - i - 2}{i} B^{k - 2i - 2} C^{i+1}
    \end{bdmatrix}.
\end{equation}
\end{corollary}
\bigbreak
This claim follows directly from \Cref{lemma:matrix_power} and the fact that the square matrices $B$ and $C$ commute. As a result, we can apply this corollary to \Cref{equation:recursion-first_order} in order to derive the following expression for the dynamics of OGDA:

\begin{equation}
    \label{equation:OGDA-solution}
    \begin{pmatrix}
    \mathbf{x}_t \\
    \mathbf{y}_t
    \end{pmatrix}
    =
    \left\{
    \sum_{i=0}^{\lfloor \frac{t}{2} \rfloor} \binom{t - i}{i} B^{t - 2i} C^i \right\}
    \begin{pmatrix}
    \mathbf{x}_0 \\
    \mathbf{y}_0
    \end{pmatrix}
    + \left\{
    \sum_{i=0}^{\lfloor \frac{t-1}{2} \rfloor} \binom{t - i - 1}{i} B^{t - 2i - 1}C^{i+1} \right\}
    \begin{pmatrix}
    \mathbf{x}_{-1} \\
    \mathbf{y}_{-1}
    \end{pmatrix}.
\end{equation}
\bigbreak
\noindent
Let us assume that

\begin{equation}
    \label{equation:convolution_Q}
    Q_t = \sum_{i=0}^{\lfloor \frac{t}{2} \rfloor} \binom{t-i}{i} B^{t-2i} C^i.
\end{equation}
Then, \Cref{equation:OGDA-solution} can be written as

\begin{equation}
    \label{equation:OGDA-solution_Q}
    \begin{pmatrix}
    \mathbf{x}_t \\
    \mathbf{y}_t
    \end{pmatrix}
    = Q_t
    \begin{pmatrix}
    \mathbf{x}_0 \\
    \mathbf{y}_0
    \end{pmatrix}
    + Q_{t-1} C
    \begin{pmatrix}
    \mathbf{x}_{-1} \\
    \mathbf{y}_{-1}
    \end{pmatrix}.
\end{equation}

Therefore, the final step is to provide a more compact expression for $Q_t$. Note that the convergence of the dynamics reduces to the convergence of $Q_t$. To this end, we will establish a generalized version of \Cref{proposition:binomial_variant} that holds for matrices. First, let

\begin{equation}
    T = B^2 + 4C =
    \begin{pmatrix}
    \mathbf{I}_n - 4\eta^2 A A^T & \mathbf{0}_{n \times m} \\
    \mathbf{0}_{m \times n} & \mathbf{I}_m - 4 \eta^2 A^T A
    \end{pmatrix}.
\end{equation}

\bigbreak
Let us assume that $\gamma = ||A|| = ||A^T||$; recall that $||A A^T|| = ||A^T A || = \gamma^2$. Naturally, we consider the non-trivial case where $\gamma \neq 0$.

\begin{lemma}
\label{lemma:positive_definite}
For any learning rate $\eta < 1/(2\gamma)$, matrix $T$ is positive definite.
\end{lemma}

We provide a proof for this claim in \Cref{subsection:proof_3}. As a result, for any sufficiently small learning rate, the (positive definite) square root of matrix $T$ -- the principal square root $T^{1/2}$ -- is well defined, as well as its inverse matrix $T^{-1/2}$.

\begin{lemma}
\label{lemma:commutative}
Consider a learning rate $\eta < 1/(2\gamma)$; then, matrices $B$ and $T^{1/2}$ commute.
\end{lemma}

\begin{corollary}
For any $\eta < 1/(2\gamma)$,
\begin{equation}
    \label{equation:solution_Q}
    Q_t = T^{-1/2} \left\{ \left( \frac{B + T^{1/2}}{2}\right)^{t+1} - \left( \frac{B - T^{1/2}}{2} \right)^{t+1} \right\}.
\end{equation}
\end{corollary}
\bigbreak
This corollary follows directly from the previous lemmas and \Cref{proposition:binomial_variant}. Therefore, if we replace the derived expression of $Q_t$ in \Cref{equation:OGDA-solution_Q} we obtain a closed-form and succinct solution for the OGDA dynamics. In the following sections we employ this result to characterize the behavior of the dynamics.

\subsection{Convergence of the Dynamics}
\label{subsection:convergence}

It is clear that analyzing the convergence of OGDA reduces to investigating the asymptotic behavior of $Q_t$; in particular, we can prove the following theorem:

\begin{theorem}
\label{theorem:convergence}
For any learning rate $\eta < 1/(2\gamma)$, OGDA converges from any initial state.
\end{theorem}

\begin{sproof}
We give a high-level sketch of our techniques. For a rigorous proof we refer to \Cref{subsection:proof_5}. First, a sufficient condition for the claim to hold is that the powers of both matrices $(B + T^{1/2})/2$ and $(B - T^{1/2})/2$ converge, which is tantamount to showing that their spectrum resides within the unit circle. In this context, although the spectrum of $B$ and $T^{1/2}$ can be determined as a simple exercise, the main complication is that the spectrum of the sum of matrices cannot be -- in general -- characterized from the individual components. Nonetheless, we show that $B$ and $T^{1/2}$ are simultaneously diagonalizable and that their eigenvalues are in a particular correspondence.
\end{sproof}

\paragraph{Remark} Throughout our analysis we made the natural assumption that the learning rate $\eta$ is positive. However, our proof of convergence in \Cref{theorem:convergence} (see \Cref{subsection:proof_5}) remains valid when $|\eta| < 1/(2\gamma)$, implying a very counter-intuitive property: OGDA can actually converge with negative learning rate! Of course, this is very surprising since performing Gradient Descent/Ascent with negative learning rate leads both players to the direction of the (locally) worst strategy. Perhaps, the optimistic term negates this intuition.

\begin{corollary}
For any learning rate $\eta < 1/(2\gamma)$, OGDA exhibits linear convergence.
\end{corollary}

We can also provide an exact characterization of the convergence's rate of OGDA with respect to the learning rate and the spectrum of the matrix of the game $A$, as stated in the following proposition.

\begin{proposition}
\label{proposition:convergence_rate}

Let $\lambda_{min}$ be the minimum non-zero eigenvalue of matrix $4 \eta^2 A A^T$; then, assuming that $\eta < 1/(2\gamma)$, the convergence rate of OGDA is $e(\lambda_{min})$, where
%$e: (0, 1] \mapsto (0, 1) $ a strictly decreasing function defined as

\begin{equation}
    e(\lambda) = \sqrt{\frac{1 + \sqrt{1 - \lambda}}{2}}.
\end{equation}

\end{proposition}

For the proof of the claim we refer to \Cref{subsection:proof}. An important consequence of this proposition is that while $\eta < 1/(2\gamma)$, increasing the learning rate will accelerate the convergence of the dynamics. Moreover, \Cref{proposition:convergence_rate} implies a rather surprising discontinuity on the rate of convergence (see \Cref{section:discontinuity}). We also refer to \Cref{section:oscillations} for a precise characterization of the inherent oscillatory component in the dynamics. This subsection is concluded with a simple example, namely $f(x,y) = x y$ with $x,y \in \mathbb{R}$. \Cref{figure:OGDA}  illustrates the impact of the learning rate to the behavior of the system.

\begin{figure}[!ht]
    \centering
    \includegraphics[scale=0.45]{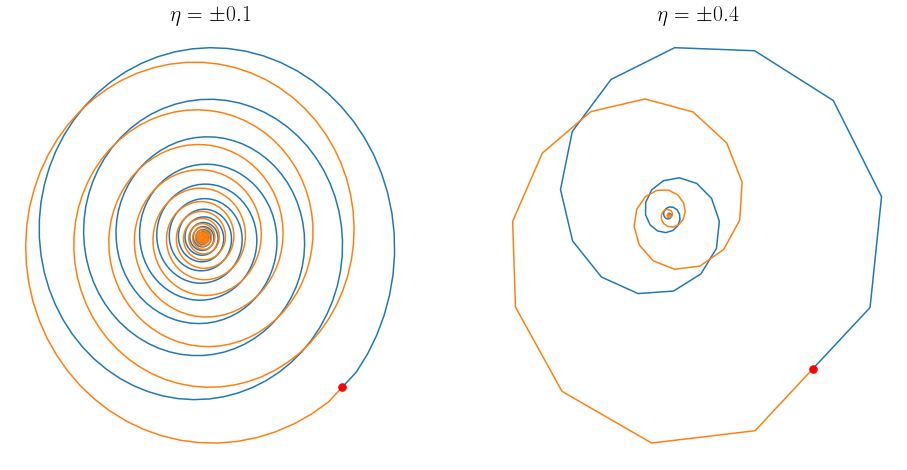}
    \caption{The trajectories of the players' strategies throughout the evolution of the game; the blue color corresponds to a positive learning rate, while the orange to a negative one. We know from \Cref{theorem:convergence} that when $\eta < 1/2$ the dynamics converge; moreover, it follows from \Cref{proposition:convergence_rate} that while the learning rate gradually increases, without exceeding the threshold of $\eta = 1/2$, the system exhibits faster convergence and limited oscillatory behavior. Remarkably, a negative learning rate simply leads to a reflection of the trajectories.}
    \label{figure:OGDA}
\end{figure}

\subsection{Limit Points of the Dynamics}
\label{subsection:limit_points}

Having established the convergence of the dynamics, the next natural question that arises relates to the characterization of the limit points. We initiate our analysis with the following proposition.

\begin{proposition}
\label{proposition:nash_unconstrained}
A pair of strategies $(\mathbf{x}^*, \mathbf{y}^*) \in \mathbb{R}^n \times \mathbb{R}^m$ constitutes a Nash equilibrium for the unconstrained bilinear  game if and only if $A \mathbf{y}^* = \mathbf{0}$ and $A^T \mathbf{x}^* = \mathbf{0}$, that is $\mathbf{y}^* \in \mathcal{N}(A)$ and $\mathbf{x}^* \in \mathcal{N}(A^T)$.
\end{proposition}

As a result, it is easy to show that when the players' strategies converge, the limit points will constitute a Nash equilibrium, as stated in the following proposition.

\begin{proposition}
\label{proposition:nash_convergence}
If OGDA converges, the limit points are Nash equilibria.
\end{proposition}

This claim follows very easily from the structure of OGDA and \Cref{proposition:nash_unconstrained}; we refer to \Cref{subsection:proof_7} for the proof. \Cref{proposition:nash_convergence} asserts that the limit points $\mathbf{x}_{\infty}$ and $\mathbf{y}_{\infty}$ reside in the left and right null space of $A$ respectively -- assuming convergence. In fact, we can establish a much more precise characterization. In particular, let $r$ the rank of matrix $A$; from the fundamental theorem of Linear Algebra it follows that $\dim \mathcal{N}(A) = m - r$ and $\dim \mathcal{N}(A^T) = n - r$. We shall prove the following theorem.

\begin{theorem}
    \label{theorem:projection}
Let $\{ \mathbf{y}^1, \mathbf{y}^2, \dots, \mathbf{y}^{m-r}\}$ an orthonormal basis for $\mathcal{N}(A)$ and $\{ \mathbf{x}^1, \mathbf{x}^2, \dots, \mathbf{x}^{n-r}\}$ an orthonormal basis for $\mathcal{N}(A^T)$; then, assuming that $\eta < 1/(2\gamma)$, $\mathbf{y}_{\infty}$ is the orthogonal projection of $\mathbf{y}_0$ to $\mathcal{N}(A)$ and $\mathbf{x}_{\infty}$ is the orthogonal projection of $\mathbf{x}_0$ to $\mathcal{N}(A^T)$, that is

\begin{equation}
    \lim_{t \to \infty} \mathbf{y}_t = \sum_{i=1}^{m-r} \langle \mathbf{y}_0, \mathbf{y}^i \rangle \mathbf{y}^i,
\end{equation}

\begin{equation}
    \lim_{t \to \infty} \mathbf{x}_{t} = \sum_{i=1}^{n-r} \langle \mathbf{x}_0, \mathbf{x}^i \rangle \mathbf{x}^i.
\end{equation}

\end{theorem}

The proof follows from our solution to the OGDA dynamics and simple calculations (see \Cref{subsection:proof-projection}). This theorem provides a very important and natural insight on the limit points of OGDA: the dynamics converge to the attractor which is closer to the initial configuration $(\mathbf{x}_0, \mathbf{y}_0)$. An important implication of this property is that we can use OGDA in order to implement the projection to the null space of $A$ -- for the $\mathbf{y}$ player -- and the null space of $A^T$ -- for the $\mathbf{x}$ player.

\section{Characterization of Nash equilibria in Constrained Games}
\label{section:nash-constrained}

The main purpose of this section is to correlate the equilibria between constrained and unconstrained games. Specifically, if $v$ represents the value of a constrained game that possesses an interior equilibrium, we show that $(\mathbf{x}, \mathbf{y}) \in \Delta_n \times \Delta_m$ is a Nash equilibrium if and only if $A \mathbf{y} = v \mathbf{1}$ and $A^T \mathbf{x} = v \mathbf{1}$; thus, when $v = 0$ the stationary points will derive from the intersection of the null space of $A$ and $A^T$ with the corresponding probability simplexes, establishing a clear connection with the saddle-points of the corresponding unconstrained bilinear game (recall \Cref{proposition:nash_unconstrained}). First, let us introduce the following notation:

\begin{equation}
    \NE = \{ (\mathbf{x}^*, \mathbf{y}^*) \in \Delta_n \times \Delta_m : (\mathbf{x}^*, \mathbf{y}^*) \text{ is an equilibrium of the constrained game} \}.
\end{equation}

Consider two pairs of Nash equilibria; a very natural question is whether a mixture -- a convex combination -- of these pairs will also constitute a Nash equilibrium. This question is answered in the affirmative in the following propositions.

\begin{proposition}
    \label{proposition:decoupling}
There exist $\Delta_n^* \subseteq \Delta_n$ and $\Delta_m^* \subseteq \Delta_m$, with $\Delta_n^*, \Delta_m^* \neq \emptyset$, such that $\NE = \Delta_n^* \times \Delta_m^*$.
\end{proposition}

\begin{sproof}
The claim follows from a simple decoupling argument; in particular, it is easy to see that

\begin{equation}
    \label{equation:decoupling}
    (\mathbf{x}^*, \mathbf{y}^*) \in \NE \iff
    \left\{
    \begin{split}
    \mathbf{x}^* \in \argmin_{\mathbf{x} \in \Delta_n} \max_{\mathbf{y} \in \Delta_m} \mathbf{x}^T A \mathbf{y} \\
    \mathbf{y}^* \in \argmax_{\mathbf{y} \in \Delta_m} \min_{\mathbf{x} \in \Delta_n} \mathbf{x}^T A \mathbf{y}
    \end{split}
    \right\}
    \iff
    \left\{
    \begin{split}
    \mathbf{x}^* \in \Delta_n^* \\
    \mathbf{y}^* \in \Delta_m^*
    \end{split}
    \right\},
\end{equation}
where $\Delta_n^*$ and $\Delta_m^*$ constitute the set of optimal strategies, as defined in \eqref{equation:decoupling}. Finally, note that Von Neumann's min-max theorem implies that $\Delta_n^*$ and $\Delta_m^*$ are non-empty.

\end{sproof}

\begin{proposition}
    \label{proposition:convexity}
The sets of optimal strategies $\Delta_n^*$ and $\Delta_m^*$ are convex.
\end{proposition}

\begin{proof}
Let $\mathbf{x}_1^*, \mathbf{x}_2^* \in \Delta_n^*$ and some $\lambda \in [0,1]$. It suffices to show that $\lambda \mathbf{x}_1^* + (1-\lambda) \mathbf{x}_2^* \in \Delta_n^*$; the convexity of $\Delta_m^*$ will then follow from an analogous argument. Indeed, we have that

\begin{align*}
    \max_{\mathbf{y} \in \Delta_m} (\lambda \mathbf{x}_1^* + (1-\lambda) \mathbf{x}_2^*)^T A \mathbf{y} &\leq \lambda \max_{\mathbf{y} \in \Delta_m} (\mathbf{x}_1^*)^T A \mathbf{y} + (1 - \lambda) \max_{\mathbf{y} \in \Delta_m} (\mathbf{x}_2^*)^T A \mathbf{y} \\
    &\leq \max_{\mathbf{y} \in \Delta_m} \mathbf{x}^T A \mathbf{y}, \quad \forall \mathbf{x} \in \Delta_n,
\end{align*}
where the last line follows from $\mathbf{x}_1^*, \mathbf{x}_2^* \in \Delta_n^*$; thus, $\lambda \mathbf{x}_1^* + (1 - \lambda) \mathbf{x}_2^* \in \Delta_n^*$.
\end{proof}

While the previous statements hold for any arbitrary game, the following proposition will require two separate hypotheses; after we establish the proof we will illustrate that these assumptions are -- in principle -- generic, in the sense that every game can be reduced with simple operations to the assumed form, without -- essentially -- altering the space of equilibria.

\begin{proposition}
    \label{proposition:NE}
Consider a constrained zero-sum game with an interior equilibrium and value $v = 0$; then, it follows that
\begin{equation}
    \NE = \left( \Delta_n \cap \mathcal{N}(A^T) \right) \times \left( \Delta_m \cap \mathcal{N}(A) \right).
\end{equation}

\end{proposition}

\paragraph{Reduction}

Finally, we shall address the assumptions made in the last proposition. First, consider a matrix $A'$ that derives from $A$ with the addition of a constant $c$ in every entry; then, for $\mathbf{x} \in \Delta_n$ and $\mathbf{y} \in \Delta_m$, it follows that

\begin{equation}
    \mathbf{x}^T A' \mathbf{y} = \mathbf{x}^T (A + c \mathbf{1}_{n \times m}) \mathbf{y} = \mathbf{x}^T A \mathbf{y} + c.
\end{equation}

Therefore, the addition of $c$ in every entry of the matrix does not alter the space of Nash equilibria; it does, however, change the value of the game by an additive constant $c$. As a result, if we add in every entry of a matrix the value $-v$, the game that arises has a value of zero. As a corollary of \Cref{proposition:NE}, $(\mathbf{x}, \mathbf{y})$ is a Nash equilibrium pair -- in a game where there exists an interior equilibrium -- if and only if $A \mathbf{y} = v \mathbf{1}$ and $A^T \mathbf{x} = v \mathbf{1}$. Furthermore, consider a game that does not possess an interior equilibrium. In particular, let us assume that every $\mathbf{x}^* \in \Delta_n^*$ resides in the boundary of $\Delta_n$. Then, it follows that there exists some action for player $\mathbf{x}$ with zero probability for every $\mathbf{x}^* \in \Delta_n^*$. Indeed, if we posit otherwise \Cref{proposition:convexity} implies the existence of an interior probability vector $\mathbf{x}^* \in \Delta_n^*$, contradicting our initial hypothesis. Thus, we can iteratively remove every row from the matrix that corresponds to such actions, until no such row exists. It is easy to see that this operation does not have an impact on the Nash equilibria of the game, modulo some dimensions in which the player always assigns zero probability mass. Finally, we can apply -- if needed -- a similar process for the column player and obtain a game with an interior equilibrium. In other words, games without interior equilibria are reducible, in the sense that the 'unsound' actions from either the row or the column player can be removed without altering the game under optimal play.

\paragraph{Remark} It is important to point out that this reduction argument is merely a thought experiment, illustrating that the games we consider capture -- at least in some sense -- the complexity of the entire class; yet, it does not provide an appropriate reduction since unknown components -- such as the value of the game -- are used. We leave as an open question whether it is possible to incorporate this reduction in our iterative algorithm.

\section{Alternating Projections}
\label{section:alternating_projections}

Based on the properties we have established in the previous sections, we provide an algorithm that (strongly) converges to the space of Nash equilibria in constrained games. Throughout this section, we will require that the value of the game is zero and that there exists an interior equilibrium (see \Cref{section:necessity} on why these assumptions are necessary); under these hypotheses, we have a strong characterization of the optimal strategies (\Cref{proposition:NE}): $\Delta_n^* = \Delta_n \cap \mathcal{N}(A^T)$ and $\Delta_m^* = \Delta_m \cap \mathcal{N}(A)$. As a result, we can employ an Alternating Projections scheme.

\paragraph{Complexity of Projection} It is important to point out that the complexity of implementing a projection to a non-empty, convex and closed set $C$ depends primarily on the structure of $C$. Indeed, identifying the projection to an arbitrary set constitutes a non-trivial optimization problem of its own. For this reason, we need OGDA to implement the projection to the null space of $A$ and $A^T$. On the other hand, we make the standard assumption that the projection to the probability simplex is computed in a single iteration.

\bigbreak
In this context, our main algorithm constitutes a projected variant of OGDA; specifically, instead of projecting after a single step, we simulate multiple \emph{steps} $T_s$ of the OGDA algorithm, before finally performing the projection. This process will be repeated for $T_p$ \emph{cycles}. We will assume for simplicity that OGDA is initialized with $\mathbf{x}_{-1} = \mathbf{0}$ and $\mathbf{y}_{-1} = \mathbf{0}$, without having an impact on its limit points (\Cref{theorem:projection}). Recall from \Cref{theorem:projection} that OGDA converges to the projection of the initial conditions $\mathbf{x}_0$ and $\mathbf{y}_0$ to the left and right null space of $A$ respectively. Thus, our algorithm essentially performs -- for sufficiently large $T_s$ -- alternate projections in $\mathcal{N}(A^T)$ and $\Delta_n$ in the domain of the $\mathbf{x}$ player, and in $\mathcal{N}(A)$ and $\Delta_m$ in the domain of the $\mathbf{y}$ player. In this sense, it can be viewed as an approximate instance of Alternating Projections. We state the following theorem in order to illustrate the robustness of this method to arbitrary Hilbert spaces.

\begin{algorithm}[h]
\caption{Alternating Projections}
\label{algorithm:projected-OGDA}
\SetKwInput{KwInput}{Input}
\SetAlgoLined
\KwResult{Approximate Nash equilibrium}
 \KwInput{matrix of the game $A$, $(\mathbf{x}_0^p, \mathbf{y}_0^p) \in \Delta_n \times \Delta_m$, learning rate $\eta$}
 \For{$k := 1, 2, \dots, T_p$}{
    $\mathbf{x}_0 := \mathbf{x}_{k-1}^p$\;
    $\mathbf{y}_0 := \mathbf{y}_{k-1}^p$\;
    \For{$t := 1, 2, \dots, T_s$} {
      $\mathbf{x}_t := \mathbf{x}_{t-1} - 2\eta A \mathbf{y}_{t-1} + \eta A \mathbf{y}_{t-2}$\;
      $\mathbf{y}_t := \mathbf{y}_{t-1} + 2\eta A^T \mathbf{x}_{t-1} - \eta A^T \mathbf{x}_{t-2}$\;
  }
  $\mathbf{x}_k^p := \mathcal{P}_{\Delta_n}(\mathbf{x}_{T_s}^{\phantom{p}})$\;
  $\mathbf{y}_k^p := \mathcal{P}_{\Delta_m}(\mathbf{y}_{T_s}^{\phantom{p}})$\;
 }
 \Return $(\mathbf{x}_{T_p}^p, \mathbf{y}_{T_p}^p)$\;
\end{algorithm}

\begin{theorem}
    \label{theorem:alternating_projections}
    Let $M$ and $N$ denote closed and convex subsets of a Hilbert space $\mathcal{H}$ with non-empty intersection and $I \in \mathcal{H}$. If one of the sets is compact \footnote{We refer to \cite{Polak1990MethodOS,10.1007/BF01027691} for more refined conditions}, the method of Alternating Projections converges in norm to a point in the intersection of the sets; that is, $\exists L \in M \cap N$, such that
    \begin{equation}
    \lim_{n \to \infty} || \left( \mathcal{P}_{M} \mathcal{P}_N \right)^n (I) - L || = 0.
    \end{equation}
\end{theorem}

This theorem directly implies the convergence of Algorithm \ref{algorithm:projected-OGDA}. Indeed, first note that a probability simplex is compact and convex. Moreover, the null space of any matrix is a vector (or linear) space of finite dimension and thus, it is trivially a convex and closed set. Finally, under our assumptions, $\Delta_n \cap \mathcal{N}(A^T) \neq \emptyset $ and $\Delta_m \cap \mathcal{N}(A) \neq \emptyset$; as a result, the following corollary follows from \Cref{theorem:projection}, \Cref{proposition:NE} and \Cref{theorem:alternating_projections}.

\begin{corollary}
    \label{corollary:convergence}
    Consider a constrained zero-sum game with an interior equilibrium and value $v = 0$; if $\eta < 1/(2\gamma)$, then for $T_s, T_p \to \infty$, Algorithm \ref{algorithm:projected-OGDA} converges in norm to a Nash equilibrium.
\end{corollary}

The final step is to characterize the rate of convergence of our algorithm. In the following proof we use techniques from \cite{10.1007/BF01027691}.

\begin{theorem}
    \label{theorem:fast_convergence}
    Consider a constrained zero-sum game with an interior equilibrium and value $v = 0$, and let $\lambda \in (0,1)$ the rate of convergence of OGDA; if $\eta < 1/(2\gamma)$, then there exists a parameter $\alpha \in (0,1)$ such that Algorithm \ref{algorithm:projected-OGDA} yields an $\epsilon$-approximate Nash equilibrium, for any $\epsilon > 0$, with $T_p \in \mathcal{O}\left( \frac{\log (1/\epsilon)}{\log (1/\alpha)} \right)$ cycles and $T_s \in \left( \frac{\log(1/((1-\alpha)\epsilon))}{\log (1/\lambda)} \right)$ steps.
\end{theorem}

\begin{proof}

First, we analyze the dynamics in the domain of player $\mathbf{y}$. Let $\mathbf{y}_i^p \equiv I \in \Delta_m$ the strategy of player $\mathbf{y}$ in some cycle $i$ of the algorithm, $P = \mathcal{P}_{\mathcal{N}(A)}(I)$ and $I' = \mathcal{P}_{\Delta_m}\mathcal{P}_{\mathcal{N}(A)} (I) = \mathcal{P}_{\Delta_m}(P)$. It is easy to see that $\exists \kappa \geq 1$, such that for any $I \in \Delta_m, d(I, \Delta_m^*) \leq \kappa d(I, \mathcal{N}(A))$. Indeed, if we assume otherwise it follows that there exists an accumulation point $I$ of $\mathcal{N}(A)$ with $I \notin \mathcal{N}(A)$, a contradiction given that $\mathcal{N}(A)$ is a closed set. Fix any arbitrary $L \in \Delta_m^*$; since projection is firmly non-expansive it follows that

\begin{align*}
    d^2(I, \mathcal{N}(A)) &= || I - P ||^2 = || (I - L) - (\mathcal{P}_{\mathcal{N}(A)} (I) - \mathcal{P}_{\mathcal{N}(A)} (L)) ||^2 \\
    &\leq || I - L ||^2 - || \mathcal{P}_{\mathcal{N}(A)} (I) - \mathcal{P}_{\mathcal{N}(A)} (L) ||^2 \\
    &\leq || I - L ||^2 - || \mathcal{P}_{\Delta_m} \mathcal{P}_{\mathcal{N}(A)} (I) - \mathcal{P}_{\Delta_m} \mathcal{P}_{\mathcal{N}(A)} (L) ||^2 \\
    &= || I - L ||^2 - || I' - L||^2.
\end{align*}

\noindent
In particular, if we let $L = \mathcal{P}_{\Delta_m^*}(I)$, it follows that

\begin{align*}
    \frac{1}{\kappa^2} d^2(I, \Delta_m^*) &\leq || I - \mathcal{P}_{\Delta_m^*}(I) ||^2 - || I' - \mathcal{P}_{\Delta_m^*}(I)||^2 \\
    &\leq d^2(I, \Delta_m^*) - d^2(I', \Delta_m^*).
\end{align*}

\noindent
As a result, we have established the following bound:

\begin{equation}
    d(I', \Delta_m^*) \leq \sqrt{1 - \frac{1}{\kappa^2}} d(I, \Delta_m^*).
\end{equation}

Let $\alpha = \sqrt{1 - 1/\kappa^2} \in [0,1)$. The case of $\alpha = 0$ is trivial and hence, we will assume that $\alpha > 0$. Since OGDA exhibits linear convergence with rate $\lambda$, it follows that we can reach within $\delta$ distance from $P$ with $T_s \in \mathcal{O}(\log \delta/\log \lambda)$ number of steps.  Given that projection is a contraction, we obtain that $\mathbf{y}_{i+1}^p$ is within $\delta$ distance from $I'$; thus, applying the triangle inequality yields that for every cycle $i$

\begin{equation}
    d(\mathbf{y}_{i+1}^p, \Delta_m^*) \leq \alpha d(\mathbf{y}_i^p, \Delta_m^*) + \delta.
\end{equation}

\noindent
Therefore, if we perform $T_p$ cycles of the algorithm, we have that

\begin{equation}
    d(\mathbf{y}_{T_p}^p, \Delta_m^*) \leq \alpha^{T_p} d(\mathbf{y}_0^p, \Delta_m^*) + \frac{\delta}{1 - \alpha}.
\end{equation}

As a result, if $T_p \in \mathcal{O}\left( \frac{\log (1/\epsilon)}{\log (1/\alpha)} \right) $ and $T_s \in \left( \frac{\log(1/((1-\alpha)\delta))}{\log (1/\lambda)} \right)$, we obtain a probability vector $\mathbf{y}^p_{T_p}$ that resides within $\epsilon + \delta$ distance from the set of optimal strategies $\Delta_m^*$. By symmetry, $\exists \alpha' \in (0,1)$ such that after an analogous number of cycles and steps the probability vector of player $\mathbf{x}^p$ will also reside within $\epsilon + \delta$ distance from her set of optimal strategies $\Delta_n^*$. Finally, the theorem follows for $a := \max\{\alpha, \alpha' \}$ by the Lipschitz continuity of the objective function of the game.
\end{proof}

As a consequence, we can reach an $\epsilon$-approximate Nash equilibrium with $\widetilde{\mathcal{O}}(\log^2(1/\epsilon))$ iterations of Algorithm \ref{algorithm:projected-OGDA}. We remark that the crucial parameter that determines the rate of convergence in the Alternating Projections procedure is the angle between the subspaces, in the form of parameter $\alpha$ - as introduced in the previous proof. Note that this parameter can be arbitrarily close to $1$, for sufficiently high dimensions and with equilibria sufficiently close to the boundary of the simplex. We refer to \Cref{section:examples} for an additional discussion and simple examples of our algorithm.

\section{Concluding Remarks}
\label{section:concluding_remarks}

The main contributions of this paper are twofold. First, we strongly supplement our knowledge on the behavior of Optimistic Gradient Descent/Ascent in bilinear games, providing a tighter and more precise characterization. Our analysis raises several questions that need to be addressed in future research and indeed, we believe that a deeper understanding of the optimistic term and its stabilizing role in the dynamics remains -- to a large extent -- unresolved. Moreover, we proposed a variant of projected OGDA that simulates an Alternating Projections procedure and we showed a surprising near-linear convergence guarantee. Our main algorithmic result applies for games with interior equilibria and value $v = 0$. In this context, an interesting avenue for future research would be to investigate whether our algorithm can be generalized for any arbitrary game; in other words, can we somehow incorporate the simple operations of our reduction in the iterative algorithm?

\bibliography{refs.bib}

\appendix

\section{Examples on Algorithm \ref{algorithm:projected-OGDA}}
\label{section:examples}

In this section, we provide simple examples in order to illustrate the behavior of our main algorithm. In particular, we first consider the matching pennies games, defined with the following matrix:

\begin{equation}
    A = \begin{pmatrix}
    1 & -1 \\
    -1 & 1
    \end{pmatrix}.
\end{equation}

It is easy to see that this game has a unique equilibrium point, in which both players are selecting their actions uniformly at random. Note that in this case, the null space of $A$ (and similarly for $A^T$) is a one-dimensional vector space -- i.e. a line -- which is perpendicular to the probability simplex $\Delta_m$ (\Cref{fig:rotation}). Thus, the dynamics converge in a single cycle of our algorithm, as implied by \Cref{theorem:projection}. This property holds more broadly in every game with a unique and uniform equilibrium. Indeed, the line of OGDA attractors in each domain will be perpendicular to the corresponding probability simplex. On the other hand, let us consider a rotated version of matching pennies:
\begin{equation}
    A_{rot} = \begin{pmatrix}
    3 & -9 \\
    -1 & 3
    \end{pmatrix}.
\end{equation}

In this case, it is easy to verify that $\mathbf{x}^* = \left( \frac{1}{4}, \frac{3}{4}\right)$ and $\mathbf{y}^* = \left( \frac{3}{4}, \frac{1}{4} \right)$, and the optimal strategies for each player are unique. Unlike the previous example, the limit points of OGDA are not valid probability vectors and multiple cycles are required to approximate the Nash equilibrium. Note that the speed of convergence can be parameterized by $\alpha$, as introduced in the proof of \Cref{theorem:fast_convergence}. Specifically, in matching pennies -- and every game with a unique and uniform equilibrium -- $\alpha = 0$ ($\iff \kappa = 1$; see proof of \Cref{theorem:fast_convergence}) and indeed, a single cycle suffices to reach the space of equilibria. On the other hand, the rotated version induces $\alpha > 0$, decelerating the convergence. In other words, the value of $\alpha$ determines how fast the Alternating Projections procedure will converge. Intuitively, one should relate $\alpha$ to the angle between $\Delta_m$ and $\mathcal{N}(A)$ -- in the domain of player $\mathbf{y}$. Naturally, this parameter is -- in general -- different for the domain of each player and thus, the rate of convergence for the players can be asymmetrical.

\begin{figure}[!ht]
    \centering
    \includegraphics[scale=0.55]{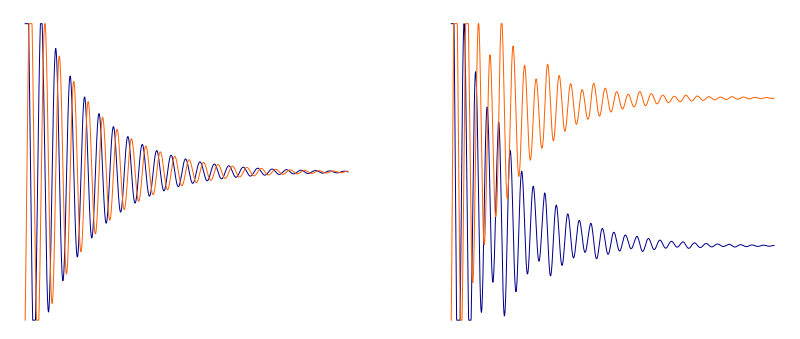}
    \caption{The trajectories of Algorithm \ref{algorithm:projected-OGDA} for matching pennies (leftmost image) and the rotated version of matching pennies (rightmost image). The blue color corresponds to the first coordinate of $\mathcal{P}_{\Delta_n}(\mathbf{x}_t)$, while the orange to the first coordinate of $\mathcal{P}_{\Delta_m}(\mathbf{y}_t)$. For every instance, the dynamics converge to the unique Nash equilibrium of the game.}
    \label{fig:projected-dynamics}
\end{figure}

\begin{figure}[!ht]
    \centering
    \includegraphics[scale=0.4]{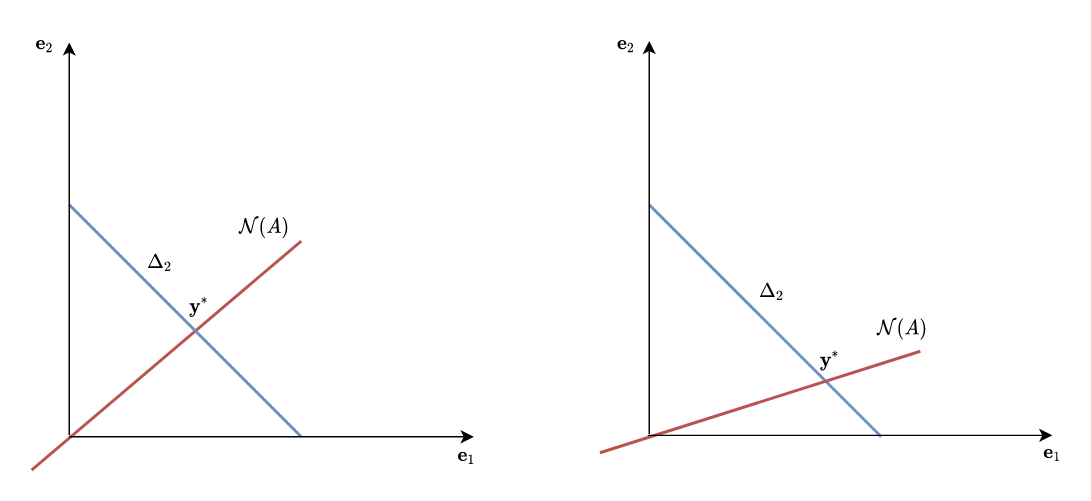}
    \caption{The corresponding geometry from \Cref{fig:projected-dynamics} in the domain of player $\mathbf{y}$. In the matching pennies game the null space $\mathcal{N}(A)$ is perpendicular to the probability simplex $\Delta_2$ (leftmost image), while rotating the matrix results in a spin of $\mathcal{N}(A)$ (rightmost image). Recall from \Cref{proposition:NE} that -- for the games we consider -- the subspace $\mathcal{N}(A)$ (and $\mathcal{N}(A^T)$) always intersects the corresponding probability simplex.}
    \label{fig:rotation}
\end{figure}

\section{Discontinuity on the Rate of Convergence of OGDA}
\label{section:discontinuity}

In this section, we elaborate on a surprising consequence of \Cref{proposition:convergence_rate} that relates to a discontinuity on the behavior of OGDA. More precisely, consider a (unconstrained) game described with the following matrix:

\begin{equation}
    A = \begin{pmatrix}
    1 & -1 \\
    -1 & 1
    \end{pmatrix}.
\end{equation}

Matrix $A A^T = A^2$ has eigenvalues at $\lambda_1 = 0$ and $\lambda_2 = 4$ and hence, as implied by \Cref{proposition:convergence_rate} the rate of convergence is $e(4\eta^2 \lambda_2)$. However, consider a perturbed matrix $A_{\epsilon}$ that derives from $A$ by the addition of a small $\epsilon > 0$ in each entry. In this case, the rate of convergence is determined by $\lambda_1' = 4 \epsilon^2$ and subsequently, it can be arbitrarily close to $1$, leading to a dramatic discontinuity in the behavior of OGDA, as illustrated in the \Cref{fig:discontinuity}. Perhaps even more remarkably, the smaller the magnitude of the noise $\epsilon$, the larger the discrepancy in the rate of convergence.

\begin{figure}[!ht]
    \centering
    \includegraphics[scale=0.52]{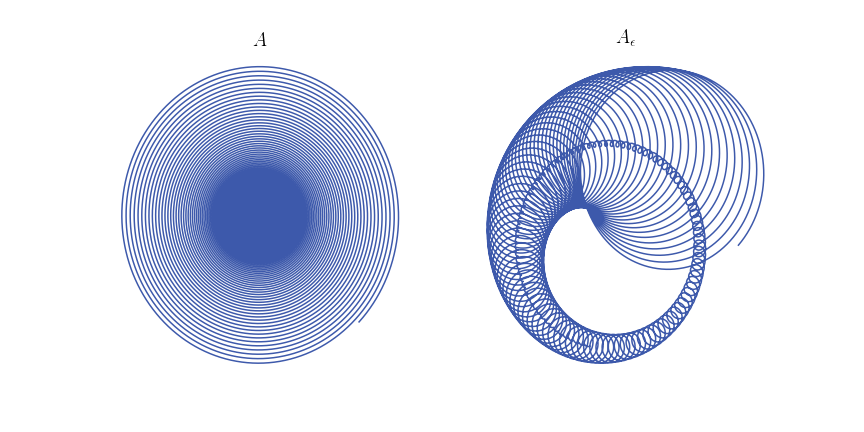}
    \caption{The dynamics of OGDA (in the first coordinate) for the same number of iterations when $\eta = 0.005$ and $\epsilon = 0.01$. Despite the presence of very limited noise, the systems exhibit a strikingly different behavior. }
    \label{fig:discontinuity}
\end{figure}

\section{Necessity of the Assumptions}
\label{section:necessity}

Our Alternating Projections algorithm (Algorithm \ref{algorithm:projected-OGDA}) requires two separate hypotheses, namely that the value of the game is zero and that there exists an interior equilibrium. It is the purpose of this section to show that both of these assumptions are also necessary. In particular, first consider the following matrix:

\begin{equation}
    A_1 = \begin{pmatrix}
    1 & 0 \\
    0 & 1
    \end{pmatrix}.
\end{equation}

It is clear that the induced game has a unique and interior equilibrium. However, since $A_1$ is square and non-singular, it follows that $\mathcal{N}(A) \cap \Delta_m = \emptyset$ and $\mathcal{N}(A^T) \cap \Delta_n = \emptyset$ and thus, the Alternating Projections algorithm will not converge. Similarly, consider the following trivial game:

\begin{equation}
    A_2 = \begin{pmatrix}
    0 & 0 \\
    1 & 1
    \end{pmatrix}.
\end{equation}

In this case, the minimizer or row player will completely disregard the second row and thus, the value of the game will be zero. However, it again follows that $\mathcal{N}(A) \cap \Delta_m = \emptyset$ and $\mathcal{N}(A^T) \cap \Delta_n = \emptyset$. Note that according to our reduction, we simply subtract $-1/2$ from each entry of $A_1$ and we remove the second row from $A_2$; after these simple operations we are able to apply our algorithm.

\section{Oscillations in OGDA}
\label{section:oscillations}

Despite the convergence guarantee of the OGDA dynamics, the trajectories of the players' strategies feature oscillations, albeit with an ever decreasing amplitude; however, our derived solution in \Cref{equation:solution_Q} does not provide an immediate justification for this phenomenon. For this reason, we shall reformulate the solution of the dynamics in a way that will explain the oscillations that occur. In particular, note that the term $Q_t$ \eqref{equation:convolution_Q} can be seen as a weighted convolution (see \Cref{section:binomial} for an asymptotic analysis of the weights) between matrices $B$ and $C$. In this section, we characterize the individual terms in the summation and then, $Q_t$ will be understood as the superposition of these atomic expressions. First, we can prove by induction the following lemma:

\begin{lemma}
\label{lemma:B_power}
Let us introduce the following notation:

\begin{equation}
    M_k(S) = \sum_{i=0}^{\lfloor \frac{k}{2} \rfloor} \binom{k}{2i} (-4 \eta^2 S S^T)^i,
\end{equation}

\begin{equation}
    N_k(S) = \sum_{i=0}^{\lfloor \frac{k-1}{2} \rfloor} \binom{k}{2i + 1} (-4 \eta^2 S S^T)^i.
\end{equation}
Then, for every $k \geq 1$,

\begin{equation}
    \label{equation:B_power}
    B^k
    =
    \begin{bdmatrix}
    M_k(A) & -2\eta A N_k(A^T) \\
    2\eta A^T N_k(A) & M_k(A^T)
    \end{bdmatrix}.
\end{equation}

\end{lemma}

We refer to \Cref{subsection:proof_8} for the proof. Importantly, $B^k$ is inherently characterized by oscillatory behavior, as stated in the following proposition.

\begin{proposition}
\label{proposition:oscillations}
For any $\eta < 1/(2\gamma)$, it follows that

\begin{equation}
    M_k(A) = (\mathbf{I}_n + 4\eta^2 A A^T)^{k/2} \cos(k \Omega),
\end{equation}

\begin{equation}
    2 \eta (A A^T)^{1/2} N_k(A) = (\mathbf{I}_n + 4\eta^2 A A^T)^{k/2} \sin(k \Omega),
\end{equation}
where $\Omega = \tan^{-1}\left( 2\eta (A A^T)^{1/2} \right)$.

\end{proposition}

Note that trigonometric functions can be defined for matrices through their Taylor expansion. It is clear that the powers of matrix $C$ will be in a pseudo-diagonal block form and hence, each matrix multiplication $B^{t-2i} C^i$ will induce a multi-dimensional, oscillating waveform with a particular frequency pattern in each dimension. In other words, $Q_t$ can be expressed as a weighted combination of sine and cosine terms with a specific frequency spectrum. Thus, the oscillatory behavior of the trajectories is inherent in the dynamics, while at the same time \Cref{theorem:convergence} reassures us that the cumulative interference of these terms will gradually vanish -- for sufficiently small learning rate -- as time progresses.

\section{Omitted Proofs}
\label{section:omitted_proofs}

\subsection{Proof of \texorpdfstring{\Cref{lemma:matrix_power}}{}}
\label{subsection:proof_1}

\begin{proof}
First, let $\chi_R(\lambda) = \det(\lambda \mathbf{I}_2 - R)$ the characteristic polynomial of matrix $R$; it follows that $\chi_R(\lambda) = \lambda^2 - b \lambda - c$. If we apply the Cayley-Hamilton theorem we obtain that

\begin{equation}
    \label{equation:cayley-hamilton}
    R^2 = b R + c \mathbf{I}_2.
\end{equation}
Inductively, it follows that

\begin{equation}
    \label{equation:matrix_power_aux}
    R^k = p_{k}(b, c) R + q_{k}(b, c) \mathbf{I}_2 \qquad k \geq 0,
\end{equation}
for some polynomials $p_k(b, c)$ and $q_k(b, c)$, with $p_0 = 0$ and $q_0 = 1$. As a result, if we multiply \eqref{equation:matrix_power_aux} by matrix $R$ and invoke \Cref{equation:cayley-hamilton}, it follows that

\begin{align}
    R^{k+1} &= p_{k}(b, c) R^2 + q_k(b, c) R \notag \\
            &= p_k(b, c) (b R + c \mathbf{I}_2) + q_k(b, c) R \notag \\
            &= (b p_k(b, c) + q_k(b, c)) R + c p_k(b, c) \mathbf{I}_2.
\end{align}

Moreover, it follows that the last equation is equivalent to the following system of polynomial equivalencies:

\begin{equation}
    \label{equation:polynomial_system}
    \left\{
    \begin{split}
    p_{k+1}(b, c) &= b p_k(b, c) + q_k(b, c) \\
    q_{k+1}(b, c) &= c p_k(b, c)
    \end{split}
    \right\}.
\end{equation}

\noindent
Finally, we can easily decouple $p_k$ from system \eqref{equation:polynomial_system}, obtaining that

\begin{equation}
    \label{equation:polynomial_recursion}
    p_{k+1}(b, c) = b p_k(b, c) + c p_{k-1}(b, c).
\end{equation}

Note that $p_0 = 0$ and $p_1 = 1$. Therefore, we can solve \Cref{equation:polynomial_system} with simple combinatorial arguments; to be more precise, imagine that we commence from a height of $k+1$. The polynomial recursion informs us on the feasible moves we can perform while decreasing the height and traversing towards the base -- height of $1$. It is clear that each path corresponds to a single monomial, while its coefficient is determined by the number of paths that lead to the same monomial. For instance, the term $b^k$ can only appear in the path that consists of unit steps. In general, the solution will comprise of monomials of the form $b^{k-2i} c^i$, where its coefficient coincides with the number of paths that lead to the same term, i.e. the number of ways we can select $i$ junctions -- where we apply a step of size $2$ -- from a total of $k-i$ moves. Thus, we have proven the following:

\begin{equation}
    \label{equation:p_solution}
    p_{k+1}(b, c) = \sum_{i = 0}^{\lfloor \frac{k}{2} \rfloor} \binom{k - i}{i} b^{k-2i}c^{i} \qquad k \geq 0,
\end{equation}

\begin{equation}
    \label{equation:q_solution}
    q_{k+1}(b, c) = \sum_{i=0}^{\lfloor \frac{k-1}{2} \rfloor} \binom{k - i - 1}{i} b^{k - 2i - 1} c^{i+1} \qquad k \geq 1.
\end{equation}

Note that if $b = c = 1$, then the polynomial recursion \eqref{equation:polynomial_recursion} induces the Fibonacci sequence and as a result, if $F_n$ denotes the $n^{\text{th}}$ term of the Fibonacci sequence, we have proven the following well-known identity:

\begin{equation}
    F_{n+1} = \sum_{i=0}^{\lfloor \frac{n}{2} \rfloor} \binom{n - i}{i}.
\end{equation}

An additional observation is that if we substitute the expressions of $p_{k}(b,c)$ and $q_k(b,c)$ in the identity $p_{k+1}(b,c) = b p_k(b,c) + q_k(b,c$, we obtain the following non-trivial formula:

\begin{equation}
    \sum_{i=0}^{\lfloor \frac{k-1}{2} \rfloor} \binom{k-i-1}{i} b^{k-2i} c^i + \sum_{i=0}^{\lfloor \frac{k-2}{2} \rfloor} \binom{k-i-2}{i}b^{k-2i-1}c^{i+1} = \sum_{i=0}^{\lfloor \frac{k}{2} \rfloor} \binom{k-i}{i} b^{k-2i} c^i.
\end{equation}

Finally, the lemma follows if we substitute the expressions of $p_k(b,c)$ \eqref{equation:p_solution} and $q_k(b,c)$ \eqref{equation:q_solution} to \Cref{equation:matrix_power_aux}.

\end{proof}

\subsection{Proof of \texorpdfstring{\Cref{proposition:binomial_variant}}{}}
\label{subsection:proof_2}

\begin{proof}

We consider again the matrix $R \in \mathbb{R}^{2 \times 2}$ defined as

$$
R =
\begin{pmatrix}
b & c \\
1 & 0
\end{pmatrix}.
$$

\noindent
Given that $b^2 + 4c > 0$, matrix $R$ has two real and distinct eigenvalues:

\begin{equation}
    \lambda_{1, 2} = \frac{b \pm \sqrt{b^2 + 4c}}{2}.
\end{equation}

Moreover, if we denote with $\mathbf{u}_1$ and $\mathbf{u}_2$ the eigenvectors that correspond to the eigenvalues $\lambda_1$ and $\lambda_2$ respectively, it follows that

\begin{equation}
    U =
    \begin{pmatrix}
    \mathbf{u}_1 & \mathbf{u}_2
    \end{pmatrix}
    =
    \begin{pmatrix}
    b + \sqrt{b^2 + 4c} & b - \sqrt{b^2 + 4c} \\
    2 & 2
    \end{pmatrix},
\end{equation}

\begin{equation}
    U^{-1} = \frac{1}{4 \sqrt{b^2 + 4c}}
    \begin{pmatrix}
    2 & -b + \sqrt{b^2 + 4c} \\
    -2 & b + \sqrt{b^2 + 4c}
    \end{pmatrix}.
\end{equation}

If $\Lambda$ denotes the diagonal matrix with entries the eigenvalues of matrix $R$, the spectral decomposition yields that

\begin{equation}
    R = U \Lambda U^{-1}.
\end{equation}
As a result, we have that

\begin{equation}
    R^k = \frac{1}{4\sqrt{b^2 + 4c}}
    \begin{pmatrix}
    b + \sqrt{b^2 + 4c} & b - \sqrt{b^2 + 4c} \\
    2 & 2
    \end{pmatrix}
    \begin{pmatrix}
    \lambda_1^k & 0 \\
    0 & \lambda_2^k
    \end{pmatrix}
    \begin{pmatrix}
    2 & -b + \sqrt{b^2 + 4c} \\
    -2 & b + \sqrt{b^2 + 4c}
    \end{pmatrix}.
\end{equation}

Thus, if we perform the matrix multiplications and equate the derived expression with the result in \Cref{lemma:matrix_power}, the claim follows.

\end{proof}

We should also point out an interesting connection between the previous formula and the Fibonacci series; in particular, let $\phi = (1 + \sqrt{5})/2$ denote the golden ratio. If we apply the identity in \Cref{proposition:binomial_variant} for $b = c = 1$, we derive the well-known closed-form solution of the Fibonacci series.

\begin{equation}
    F_{n+1} = \sum_{i=0}^{\lfloor \frac{n}{2} \rfloor} \binom{n-i}{i} = \frac{1}{\sqrt{5}} \left( (\phi)^{n+1} - (-\phi^{-1})^{n+1}  \right).
\end{equation}

\subsection{Proof of \texorpdfstring{\Cref{lemma:positive_definite}}{}}
\label{subsection:proof_3}

\begin{proof}

The claim is tantamount to showing that the matrices $\mathbf{I}_n - 4\eta^2 A A^T$ and $\mathbf{I}_m - 4\eta^2 A^T A$ are both positive definite. We shall prove the claim for the former; the latter admits an analogous proof. Consider a non-zero vector $\mathbf{x} \in \mathbb{R}^n$; then, the quadratic form of the matrix can be written as

\begin{equation}
    \label{equation:quadratic_form}
    \mathbf{x}^T (\mathbf{I}_n - 4\eta^2 A A^T) \mathbf{x} = ||\mathbf{x}||^2 - 4\eta^2 ||A^T\mathbf{x}||^2.
\end{equation}

Moreover, since $\eta < 1/(2\gamma)$, it follows that $4\eta^2 ||A^T \mathbf{x}||^2 \leq 4\eta^2 ||A^T||^2 ||\mathbf{x}||^2 < ||\mathbf{x}||^2$; hence, the quadratic form is always positive and the lemma follows.

\end{proof}

\subsection{Proof of \texorpdfstring{\Cref{lemma:commutative}}{}}
\label{subsection:proof_4}

\begin{proof}
First, we will show that the $\rho(T) \leq 1$. Indeed, note that the eigenvalues of $T$ are the union of the eigenvalues of matrices $\mathbf{I}_n - 4\eta^2 A A^T$ and $\mathbf{I}_m - 4\eta^2 A^T A$. Let $\chi(\lambda)$ denote the characteristic equation of matrix $ 4\eta^2 A A^T$; the roots of $\chi(\lambda)$ are real numbers in $[0, 1)$. In addition,

\begin{equation}
    \det(\lambda \mathbf{I}_n - (\mathbf{I}_n - 4\eta^2 A A^T)) = \chi(1 - \lambda).
\end{equation}

Thus, the last equation implies that the eigenvalues of the matrix $\mathbf{I}_n - 4\eta^2 A A^T$ reside in $(0, 1]$. A similar argument holds for the matrix $\mathbf{I}_m - 4\eta^2 A^T A$ and thus, $\rho(T) \leq 1$. As a result, given that the eigenvalues of $T$ are real numbers, the following power series converges:

\begin{equation}
    T^{1/2} = \mathbf{I}_{n + m} - \sum_{i=1}^{\infty} \left| \binom{1/2}{i} \right| (\mathbf{I}_{n + m} - T)^i.
\end{equation}

Moreover, note that $T$ and $B$ commute since $B$ and $C$ commute. Thus, applying the previous formula completes the proof.
\end{proof}

\subsection{Proof of \texorpdfstring{\Cref{theorem:convergence}}{}}
\label{subsection:proof_5}

The proof of the theorem requires analyzing the spectrum of the matrices that appear in the expression of $Q_t$; first, recall that a real and square matrix $S$ is said to be \emph{normal} if $S S^T = S^T S$. It is easy to see that $T^{1/2}$ is symmetric and hence it will also satisfy the aforementioned property. Moreover, it follows that $B$ is also a normal matrix:

\begin{equation}
    B B^T = \begin{pmatrix}
    \mathbf{I}_n + 4\eta^2 A A^T & \mathbf{0}_{n \times m} \\
    \mathbf{0}_{m \times n} & \mathbf{I}_m + 4\eta^2 A^T A
    \end{pmatrix}
    = B^T B.
\end{equation}

In addition, we know from \Cref{lemma:commutative} that $T^{1/2}$ and $B$ are commutative, that is $T^{1/2} B = B T^{1/2}$. As a result, we know from a fundamental theorem of Linear Algebra that $B$ and $T^{1/2}$ are simultaneously unitarily diagonalizable, that is there exists a unitary matrix $U$, such that the matrices $U^* T^{1/2} U$ and $U^{*} B U$ are diagonal -- where $U^*$ denotes the conjugate transpose of $U$. Let $\mathbf{u}$ a column of matrix $U$ and $\mathbf{u}_1$ and $\mathbf{u}_2$ the sub-vectors of dimension $n$ and $m$ respectively. Then, it follows that $\mathbf{u}$ constitutes an eigenvector of $T^{1/2}$ and $B$. Moreover, it is easy to see that $\mathbf{u}$ will also be an eigenvector for $T$; if we let $\mu$ and $\mu'$ the corresponding eigenvalues for $T$ and $B$ respectively, we have that

\begin{equation}
    \label{equation:eigenvectors_T}
    T \mathbf{u} = \mu \mathbf{u} \iff
    \left\{
    \begin{split}
    4\eta^2 A A^T \mathbf{u}_1 &= (1 - \mu)\mathbf{u}_1 \\
    4\eta^2 A^T A \mathbf{u}_2 &= (1 - \mu)\mathbf{u}_2
    \end{split}
    \right\},
\end{equation}

\bigbreak

\begin{equation}
    \label{equation:eigenvectors_B}
    B \mathbf{u} = \mu' \mathbf{u} \iff
    \left\{
    \begin{split}
    2\eta A \mathbf{u}_2 = (1 - \mu') \mathbf{u}_1 \\
    2\eta A^T \mathbf{u}_1 = (\mu' - 1) \mathbf{u}_2
    \end{split}
    \right\}.
\end{equation}

\noindent
Therefore, if we combine these equations we can deduce that

\begin{equation}
    \left\{
    \begin{split}
    (1 - \mu) \mathbf{u}_1 = - (1 - \mu')^2 \mathbf{u}_1 \\
    (1 - \mu) \mathbf{u}_2 = - (1 - \mu')^2 \mathbf{u}_2
    \end{split}
    \right\}.
\end{equation}

\bigbreak
Thus, since $\mathbf{u} \neq \mathbf{0}$ it follows that $(1 - \mu) = - (1 - \mu')^2$. Moreover, let us denote with $\Lambda$ and $\overline{\Lambda}$ the set of eigenvalues of $4\eta^2 A A^T$ and $4\eta^2 A^T A$ respectively; since, $A A^T$ and $A^T A$ are symmetric positive semi-definite matrices and $\eta < 1/(2\gamma)$, it follows that $\lambda \in [0,1), \forall \lambda \in \Lambda \cup \overline{\Lambda}$. It is easy to verify that the set of eigenvalues of $T$ will be $\{ 1 - \lambda : \lambda \in \Lambda \cup \overline{\Lambda} \}$. As a result, an eigenvalue $\lambda \in \Lambda \cup \overline{\Lambda}$ will induce an eigenvalue $\mu'$ in matrix $B$, such that $-(1 - \mu')^2 = \lambda \iff \mu' = 1 \pm \sqrt{\lambda} j$. Note that the double multiplicity is justified by the fact that the subset of non-zero eigenvalues of $\Lambda$ coincides with the subset of non-zero eigenvalues of $\overline{\Lambda}$. As a result, if $\lambda_1, \lambda_2, \dots, \lambda_r$ the non-zero eigenvalues of $\Lambda$ -- or equivalently $\overline{\Lambda}$ -- we have proven that

\begin{equation}
    U^* T^{1/2} U = \text{diagonal}(\overbrace{\sqrt{1 - \lambda_1}, \sqrt{1 - \lambda_1}, \dots, \sqrt{1 - \lambda_r}, \sqrt{1 - \lambda_r}, 1, \dots, 1}^{n+m \text{ eigenvalues}}),
\end{equation}

\begin{equation}
    U^* B U = \text{diagonal}(\overbrace{1 + \sqrt{\lambda_1} j, 1 - \sqrt{\lambda_1} j, \dots, 1 + \sqrt{\lambda_r} j, 1 - \sqrt{\lambda_r} j, 1, \dots,  1}^{n+m \text{ eigenvalues}}).
\end{equation}

\noindent
Finally, let us introduce the following matrices:

\begin{equation}
    \label{equation:matrix_U}
    E = \frac{B + T^{1/2}}{2},
\end{equation}

\begin{equation}
    \label{equation:matrix_V}
    V = \frac{B - T^{1/2}}{2}.
\end{equation}

It is sufficient to show that the power of each of these matrices converges while $t \to \infty$. In particular, recall that $S^k$ -- with $S \in \mathcal{M}_n$ -- converges while $k \to \infty$ if and only if $\rho(S) \leq 1$ and the only eigenvalues with unit norm are real and positive numbers. Through this prism, notice that Matrix $E$ has eigenvalues of the following form:

\begin{equation}
    \label{equation:eigenvalues_U}
    \lambda_e = \frac{\sqrt{1 - \lambda} + 1 \pm \sqrt{\lambda} j}{2},
\end{equation}
where $\lambda \in [0, 1)$. Therefore, we can verify that $|\lambda_e| \leq 1$ and $|\lambda_e| = 1 \iff \lambda_e = 1$. Similarly, we can show that $|\lambda_v| < 1$. As a result, it follows that when $t \to \infty$, $V^t$ converges to the zero matrix, while $E^t$ also converges -- not necessarily to the zero matrix; thus, it follows that the dynamics of OGDA converge, completing the proof.

\subsection{Proof of \texorpdfstring{\Cref{proposition:convergence_rate}}{}}
\label{subsection:proof}

\begin{proof}

Let us assume that $\Lambda$ denotes the set of eigenvalues of $4\eta^2 A A^T$. In \Cref{subsection:proof_5}, we showed that the spectrum of matrices E \eqref{equation:matrix_U} and V \eqref{equation:matrix_V} can be described as follows:

\begin{equation}
    \tag{\ref{equation:eigenvalues_U}}
    \lambda_e = \frac{\sqrt{1 - \lambda} +1 \pm \sqrt{\lambda} j}{2},
\end{equation}

\begin{equation}
    \lambda_v = \frac{-\sqrt{1 - \lambda} + 1 \pm \sqrt{\lambda}j}{2},
\end{equation}
for $\lambda \in [0,1)$. Let $e(\lambda) = |\lambda_e|$ and $v(\lambda) = |\lambda_v|$; these functions capture the spectrum transformation from matrix $4\eta^2 A A^T$ to matrices $E$ and $V$ (see \Cref{figure:spectrum_transformation}). The first important observation is that for every $\lambda_1, \lambda_2 \in [0,1), e(\lambda_1) > v(\lambda_2)$; in words, $V^t$ converges faster than $E^t$ in the direction of every eigenvector. This implies that the convergence rate of $Q_t$ -- and subsequently the OGDA dynamics -- is determined by the rate of convergence of $E^t$. Moreover, $e(\lambda)$ is strictly decreasing in $\lambda$, while also note that the eigenvalues of $4\eta^2 A A^T$ are directly proportional to $\eta^2$. Therefore, the optimal convergence rate is obtained for the maximal value of $\eta$, that is for $\eta \uparrow 1/(2\gamma)$. Finally, let $\lambda_{min} = \min \{ \lambda \in \Lambda : \lambda \neq 0 \}$. Given that $e(\cdot)$ is a strictly decreasing function, the convergence rate of OGDA is $e(\lambda_{min}) \in (0,1)$. Note that the zero eigenvalues of $A A^T$ and $A^T A$ induce stationarity in the corresponding eigenvectors. We also remark that a lower bound for the convergence rate is $\sqrt{2}/2$, obtained when every non-zero eigenvalue of $4\eta^2 A A^T$ approaches to $1$. On the other hand, eigenvalues of $4\eta^2 A A^T$ with very small value will lead to a rate of convergence close to sub-linear (see \Cref{section:discontinuity}).

\begin{figure}[!ht]
    \centering
    \includegraphics[scale=0.4]{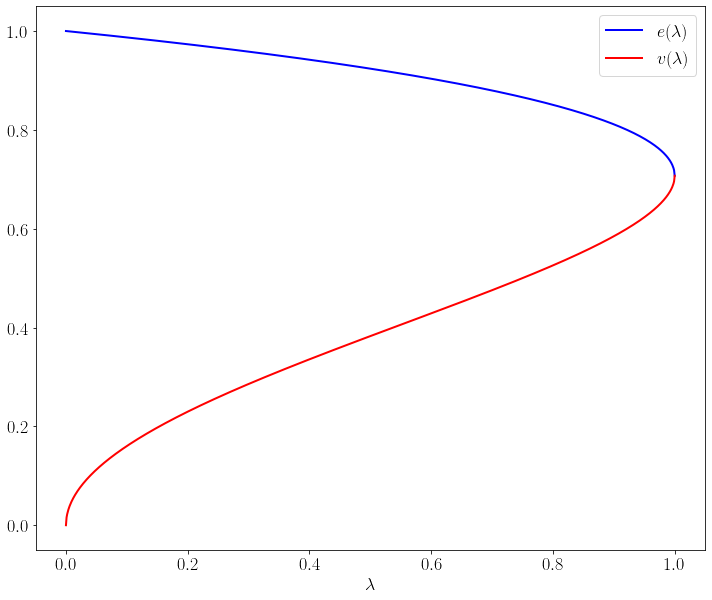}
    \caption{The spectrum transformation from matrix $4\eta^2 A A^T$ to matrices $E$ and $V$. The blue line corresponds to function $e(\lambda)$ and the red line to $v(\lambda)$. It is clear that $e(\lambda) > v(\lambda)$, while they intersect in the limit $\lambda \to 1^-$. }
    \label{figure:spectrum_transformation}
\end{figure}

\end{proof}

\subsection{Proof of \texorpdfstring{\Cref{proposition:nash_unconstrained}}{}}
\label{subsection:proof_6}

\begin{proof}

First, consider a pair of strategies $(\mathbf{x}, \mathbf{y})$ such that $A \mathbf{y} = \mathbf{0}$ and $A^T \mathbf{x} = \mathbf{0}$. We will show that $(\mathbf{x}, \mathbf{y})$ constitutes a Nash equilibrium pair. Indeed, if we fix the action of player $\mathcal{X}$, the objective function $\mathbf{x}^T A \mathbf{y}$ is nullified -- since $\mathbf{x} \in \mathcal{N}(A^T)$ -- independently from the action of player $\mathcal{Y}$; hence, deviating from $\mathbf{y}$ cannot increase her obtained utility from the game. A similar argument holds if we fix the action of player $\mathcal{Y}$.

Conversely, assume that player $\mathcal{Y}$ responds with a strategy $\mathbf{y}$ such that $A \mathbf{y} \neq \mathbf{0}$, that is there exists an index $i$ such that $e_i^T A \mathbf{y} = v \neq 0$ -- with $e_i \in \mathbb{R}^n$ the vector from the canonical orthonormal basis with $1$ in its $i^{\text{th}}$ coordinate. Then, the minimizer player $\mathcal{X}$ can respond with a strategy of $\mathbf{x} = - v e_i$ that will lead to a negative value for the objective function of $- v^2 < 0$. However, the maximizer player $\mathcal{Y}$ could deviate, responding with a vector of $\mathbf{y} = \mathbf{0}$ that increases the objective function to zero and thus improves her utility. Therefore, a vector $\mathbf{y} \notin \mathcal{N}(A)$ cannot correspond to a Nash equilibrium, and a similar argument applies for player $\mathcal{X}$. In other words, if $(\mathbf{x}, \mathbf{y})$ constitutes a Nash equilibrium, then $\mathbf{x} \in \mathcal{N}(A^T)$ and $\mathbf{y} \in \mathcal{N}(A)$.

\end{proof}

\subsection{Proof of \texorpdfstring{\Cref{proposition:nash_convergence}}{}}
\label{subsection:proof_7}

\begin{proof}
Let us denote with $\mathbf{x}_{\infty} = \lim_{t \to + \infty} \mathbf{x}_t$, and $\mathbf{y}_{\infty} = \lim_{t \to + \infty} \mathbf{y}_t,$. It follows from \Cref{equation:OGDA-bilinear} that

\begin{equation}
    \lim_{t \to + \infty} (\mathbf{x}_{t+1} - \mathbf{x}_t) = \mathbf{0} = - 2\eta A \mathbf{y}_{\infty} + \eta A \mathbf{y}_{\infty} \implies A \mathbf{y}_{\infty} = \mathbf{0},
\end{equation}

\begin{equation}
    \lim_{t \to + \infty} (\mathbf{y}_{t+1} - \mathbf{y}_t) = \mathbf{0} = - 2\eta A^T \mathbf{x}_{\infty} + \eta A^T \mathbf{x}_{\infty} \implies A^T \mathbf{x}_{\infty} = \mathbf{0}.
\end{equation}

\noindent
As a result, the claim follows from \Cref{proposition:nash_unconstrained}.

\end{proof}

\subsection{Proof of \texorpdfstring{\Cref{theorem:projection}}{}}
\label{subsection:proof-projection}

\begin{proof}

Let us denote with $Q_{\infty} = \lim_{t \to \infty} Q_t$ and introduce again the matrices $E = (B + T^{1/2})/2$ and $V=(B - T^{1/2})/2$. Moreover, let $\Lambda$ and $\overline{\Lambda}$ the set of eigenvalues of $4\eta^2 A A^T$ and $4\eta^2 A^T A$ respectively; given that these matrices are positive semi-definite and $\gamma < 1/(2\gamma)$, it follows that $\lambda \in [0, 1), \forall \lambda \in \Lambda \cup \overline{\Lambda}$. As we showed in \Cref{subsection:proof_5}, $B$ and $T^{1/2}$ are simultaneously unitarily diagonalizable, i.e. there exists a unitary matrix $U$ such that $U^* T^{1/2} U$ and $U^* B U$ are diagonal, with $U^*$ denoting the conjugate transpose of $U$. This property allowed us to identify the spectrum of matrices $E$ and $V$ from the individual spectrum of $B$ and $T^{1/2}$. In particular, it follows that $\lambda_e = (1 + \sqrt{1 - \lambda} \pm \sqrt{\lambda} j)/2$ and $\lambda_v = (1 - \sqrt{1 - \lambda} \pm \sqrt{\lambda}j)/2$, with $\lambda_e$ and $\lambda_v$ representing the eigenvalues of $E$ and $V$, $\lambda \in [0,1)$ and $j$ the imaginary unit. As a result, we obtain that $|\lambda_v| < 1$ and thus $V^t$ converges to the zero matrix, as $t \to \infty$. Similarly, in matrix $\lim_{t \to \infty} E^{t}$ every eigenvalue vanishes, except the stationary eigenvalues $\lambda_e = 1$; that is $|\lambda_e| \leq 1$ and $|\lambda_e| = 1 \iff \lambda_e = 1$.

Moreover, the eigenvectors of matrices $B$ and $T^{1/2}$ that correspond to these stationary eigenvalues serve as a basis for $\mathcal{N}(A)$ and $\mathcal{N}(A^T)$; more precisely, if $\mathbf{u}$ is a column of matrix $U$, it follows that the solution space of $B \mathbf{u} = \mathbf{u}$ coincides with the solution space of $T^{1/2} \mathbf{u} = \mathbf{u}$ ($\iff T \mathbf{u} = \mathbf{u}$). This equivalence derives from the equality of $\mathcal{N}(A)$ with $\mathcal{N}(A^T A)$ and $\mathcal{N}(A^T)$ with $\mathcal{N}(A A^T)$ (see \eqref{equation:eigenvectors_T} and \eqref{equation:eigenvectors_B} for $\mu = \mu' = 1$). Thus, since $U$ is a unitary matrix -- the columns form a orthonormal basis, it follows that the set of columns of $U$ that correspond to eigenvalues of $\lambda = 1$ can be written as

\begin{equation}
    \left\{
    \begin{pmatrix}
    \mathbf{0}_n \\
    \mathbf{y}
    \end{pmatrix}
    : \mathbf{y} \in B_y
    \right\}
    \mathsmaller{\bigcup}
    \left\{
    \begin{pmatrix}
    \mathbf{x} \\
    \mathbf{0}_m
    \end{pmatrix}
    : \mathbf{x} \in B_x
    \right\},
\end{equation}
where $B_y = \{\mathbf{y}^1, \dots, \mathbf{y}^{m-r}\}$ some orthonormal basis for $\mathcal{N}(A)$ and $B_x = \{ \mathbf{x}^1, \dots, \mathbf{x}^{n-r}\}$ some orthonormal basis for $\mathcal{N}(A^T)$. Next, note that the action of every other eigenvector asymptotically vanishes since the corresponding eigenvalue vanishes in norm. As a result, we obtain from simple algebraic calculations that

\begin{equation}
    \lim_{t \to \infty} E^t = \sum_{i=1}^{n-r}
    \begin{pmatrix}
    \mathbf{x}^i (\mathbf{x}^i)^T & \mathbf{0}_{n \times m} \\
    \mathbf{0}_{m \times n} & \mathbf{0}_{m \times m}
    \end{pmatrix}
    + \sum_{i=1}^{m-r}
    \begin{pmatrix}
    \mathbf{0}_{n \times n} & \mathbf{0}_{n \times m} \\
    \mathbf{0}_{m \times n} & \mathbf{y}^i (\mathbf{y}^i)^T
    \end{pmatrix}.
\end{equation}

In addition, the eigenvectors of $T^{1/2}$ and $T^{-1/2}$ are -- in correspondence -- equal and thus, it follows that $(\mathbf{I}_n - 4\eta^2 A A^T)^{-1/2} \mathbf{x} = \mathbf{x}, \forall \mathbf{x} \in B_x$ and similarly $(\mathbf{I}_m - 4\eta^2 A^T A)^{-1/2} \mathbf{y} = \mathbf{y}, \forall \mathbf{y} \in B_y$. Therefore, it follows that the action of $T^{-1/2}$ is idempotent with respect to $\lim_{t \to \infty} E^t:$

\begin{equation}
    \label{equation:Q-inf}
    Q_{\infty} = T^{-1/2} \lim_{t \to \infty} E^t = \sum_{i=1}^{n-r}
    \begin{pmatrix}
    \mathbf{x}^i (\mathbf{x}^i)^T & \mathbf{0}_{n \times m} \\
    \mathbf{0}_{m \times n} & \mathbf{0}_{m \times m}
    \end{pmatrix}
    + \sum_{i=1}^{m-r}
    \begin{pmatrix}
    \mathbf{0}_{n \times n} & \mathbf{0}_{n \times m} \\
    \mathbf{0}_{m \times n} & \mathbf{y}^i (\mathbf{y}^i)^T
    \end{pmatrix}.
\end{equation}

Let us return to \eqref{equation:OGDA-solution}; given that $A \mathbf{y} = \mathbf{0}, \forall \mathbf{y} \in B_y$ and $A^T \mathbf{x} = 0, \forall \mathbf{x} \in B_x$, it is easy to see that $Q_{\infty} C$ is the zero matrix. As a corollary, the initial conditions $(\mathbf{x}_{-1}, \mathbf{y}_{-1})$ do not influence the limit points of the dynamics. Finally, the theorem follows directly from \eqref{equation:Q-inf} and \eqref{equation:OGDA-solution}.

\end{proof}

\subsection{Proof of \texorpdfstring{\Cref{proposition:NE}}{}}

\begin{proof}
Let $(\mathbf{x}_i^*, \mathbf{y}_i^*)$ be an interior equilibrium of the game. Consider $\mathbf{x} \in \Delta_n$ and $\mathbf{y} \in \Delta_m$ such that $A^T \mathbf{x} = \mathbf{0}$ and $A \mathbf{y} = \mathbf{0}$; it is clear that $(\mathbf{x}, \mathbf{y})$ is a Nash equilibrium pair since any unilateral deviation from either player will not alter the expected utility from the game. For the other inclusion, we shall first show that $A \mathbf{y}_i^* = \mathbf{0}$ and $A^T \mathbf{x}_i^* = \mathbf{0}$. Indeed, if there was a single entry in vector $A \mathbf{y}_i^*$ strictly smaller than zero, then given that the strategy of the min player $\mathbf{x}_i^*$ is a best response to $\mathbf{y}_i^*$, the value of the game would be also strictly less than zero, contradicting the initial assumption. Moreover, if there was a single entry in $A \mathbf{y}_i^*$ strictly greater than zero, then the best response to $\mathbf{y}_i^*$ would not be an interior probability vector $\mathbf{x}_i^*$. An analogous argument shows that $A^T \mathbf{x}_i^* = \mathbf{0}$. Finally, consider any $(\mathbf{x}^*, \mathbf{y}^*) \in \NE$; then, it follows from \Cref{proposition:decoupling} that $(\mathbf{x}^*, \mathbf{y}_i^*) \in \NE$. Given that the value of the game is zero, every entry in $A^T \mathbf{x}^*$ is non-positive, and at the same time non-negative since the interior probability vector $\mathbf{y}_i^*$ is a best response to $\mathbf{x}^*$. Thus, $\mathbf{x}^* \in \mathcal{N}(A^T)$ and a similar argument proves that $\mathbf{y}^* \in \mathcal{N}(A)$.
\end{proof}

\subsection{Proof of \texorpdfstring{\Cref{lemma:B_power}}{}}
\label{subsection:proof_8}

\begin{proof}
The lemma will be proved through induction. First, the claim holds trivially for $k=1$. Next, we make the inductive hypothesis that

\begin{equation}
    B^k =
    \begin{bdmatrix}
    M_k(A) & -2\eta A N_k(A^T) \\
    2\eta A^T N_k(A) & M_k(A^T)
    \end{bdmatrix},
\end{equation}
for any $k \in \mathbb{N}$. It suffices to show that the claim follows for $k+1$. In particular, from the inductive hypothesis it follows that

\begin{equation}
    B^{k+1} =
    \begin{bdmatrix}
    M_k(A) - 4\eta^2 A N_k(A^T) A^T & -2\eta M_k(A) A -2\eta A N_k(A^T) \\
    2\eta A^T N_k(A) + 2\eta M_k(A^T) A^T & - 4\eta^2 A^T N_k(A) A + M_k(A^T)
    \end{bdmatrix}.
\end{equation}
\bigbreak
We will show that $M_k(A) - 4\eta^2 A N_k(A^T) A^T = M_{k+1}(A)$; the remainder of the claim will then admit an analogous proof. For the analysis we will use Pascal's formula:

\begin{equation}
    \binom{n}{k} = \binom{n-1}{k} + \binom{n-1}{k-1}.
\end{equation}

\bigbreak
\noindent
First, we consider the case where $k-1 \equiv 0 \mod 2$; then, it follows that

\begin{align}
    M_k(A) - 4\eta^2 A N_k(A^T) A^T &= \sum_{i=0}^{\lfloor \frac{k}{2} \rfloor} \binom{k}{2i} (-4\eta^2 A A^T)^i + \sum_{i=0}^{\lfloor \frac{k-1}{2} \rfloor} \binom{k}{2i+1} (-4\eta^2 A A^T)^{i+1} \\
    &= \sum_{i=0}^{\frac{k-1}{2}} \binom{k}{2i} (-4\eta^2 A A^T)^i + \sum_{i=1}^{\frac{k+1}{2}} \binom{k}{2i-1} (-4\eta^2 A A^T)^{i} \\
    &= \mathbf{I}_{n} + \sum_{i=1}^{\frac{k-1}{2}} \binom{k+1}{2i} (-4\eta^2 A A^T)^i + (-4\eta^2 A A^T)^{\frac{k+1}{2}} \\
    &= M_{k+1}(A),
\end{align}
where we have applied Pascal's formula in the third line. Similarly, if $k-1 \equiv 1 \mod 2$ it follows that

\begin{align}
    M_k(A) - 4\eta^2 A N_k(A^T) A^T &= \sum_{i=0}^{\lfloor \frac{k}{2} \rfloor} \binom{k}{2i} (-4\eta^2 A A^T)^i + \sum_{i=0}^{\lfloor \frac{k-1}{2} \rfloor} \binom{k}{2i+1} (-4\eta^2 A A^T)^{i+1} \\
    &= \sum_{i=0}^{\frac{k}{2}} \binom{k}{2i} (-4\eta^2 A A^T)^i + \sum_{i=1}^{\frac{k}{2}} \binom{k}{2i-1} (-4\eta^2 A A^T)^{i} \\
    &= \mathbf{I}_{n} + \sum_{i=1}^{\frac{k}{2}} \binom{k+1}{2i} (-4\eta^2 A A^T)^i \\
    &= M_{k+1}(A).
\end{align}

\end{proof}

\subsection{Proof of \texorpdfstring{\Cref{proposition:oscillations}}{}}
\label{subsection:proof_9}

\begin{proof}
First, recall that from the binomial theorem we have that for any $z \in \mathbb{C}$

\begin{equation}
    (1 + z)^{n} = \sum_{i=0}^n \binom{n}{i} z^i.
\end{equation}

If we apply this formula for $z = x e^{j \pi/2} = x j$ with $x \in \mathbb{R}$ and j the imaginary unit, it follows that

\begin{equation}
    \Re\left\{ (1 + x j)^n \right\} = \sum_{i=0}^{\lfloor \frac{n}{2} \rfloor} (-1)^i \binom{n}{2i} x^{2i}.
\end{equation}

\noindent
Moreover, if we express the complex number in trigonometric form we have that

\begin{equation}
    \Re\left\{ (1 + xj)^n \right\} = \Re\left\{ \left( \sqrt{1 + x^2} e^{j \tan^{-1}(x)} \right)^n \right\} = (1+x^2)^{n/2} \cos(n \tan^{-1}(x)).
\end{equation}

\noindent
As a result, we have proven that for every $x \in \mathbb{R}$

\begin{equation}
    \label{equation:identity_cos}
    \sum_{i=0}^{\lfloor \frac{n}{2} \rfloor} \binom{n}{2i} (-x^2)^{i} = (1+x^2)^{n/2} \cos(n \tan^{-1}(x)).
\end{equation}

\noindent
Similarly, we can show that

\begin{equation}
    \label{equation:identity_sin}
    x \sum_{i=0}^{\lfloor \frac{n-1}{2} \rfloor} \binom{n}{2i+1} (-x^2)^i = (1+x^2)^{n/2} \sin(n \tan^{-1}(x)).
\end{equation}

We remark that if we replace $x=1$, we obtain certain well known identities. We also stress that any trigonometric function can be defined for matrices through its Taylor expansion. Recall that the Taylor expansion of $\cos(x)$ and $\sin(x)$ converges for any $x \in \mathbb{R}$, while the Taylor expansion of $\tan^{-1}(x)$ converges for any $|x| \leq 1$; thus, $\tan^{-1}(2\eta(A A^T)^{1/2})$ will also converge for $\eta < 1/(2\gamma)$ and the lemma follows directly from \eqref{equation:identity_cos} and \eqref{equation:identity_sin}.

\end{proof}

\section{Asymptotic Analysis of Binomial Coefficient}
\label{section:binomial}

In our analysis we encountered the following binomial coefficient:

\begin{equation}
    \label{equation:binomial_term}
    \binom{n-i}{i}.
\end{equation}

In this section, we provide an asymptotic characterization of the behavior of this particular term. To be precise, it is well known that

\begin{equation}
    \binom{n}{i} \in \mathcal{O} \left( \frac{2^n}{\sqrt{n}} \right),
\end{equation}
for every $i$ such that $0 \leq i \leq n$ and the bound is asymptotically tight for $i = \lfloor n/2 \rfloor$. We derive an analogous expression for the term \eqref{equation:binomial_term}; in particular, if $\phi = (1 + \sqrt{5})/2$ denotes the golden ratio, we will prove that

\begin{equation}
    \binom{n - i}{i} \in \mathcal{O} \left( \frac{\phi^n}{\sqrt{n}} \right).
\end{equation}

We commence the analysis with a simple lemma that will be directly applied to prove the aforementioned bound.

\begin{lemma}
\label{lemma:aux}
Consider a function $f: (0, \frac{1}{2}) \mapsto \mathbb{R}^+$ defined as

\begin{equation}
    f(a) = \frac{(1-a)^{1-a}}{a^a(1-2a)^{1-2a}}.
\end{equation}
Then, $f$ obtains its maximum for $a^* = (5 - \sqrt{5})/10$, with $f(a^*) = \phi$.

\end{lemma}

\begin{proof}
Consider the function $g(a) = \log f(a)$, which can be expressed as follows:

\begin{equation}
    g(a) = (1-a) \log (1-a) - a \log a - (1-2a) \log (1-2a).
\end{equation}

It is easy to check that $g''(a) < 0$ and hence $g$ is concave; moreover, it is trivial to show that there exists a point $a^*$ that maximizes $g$ and thus, it follows that

\begin{equation}
    \log(1-2a^*) = \frac{\log(1-a^*) + \log a^*}{2} \iff 1 - 2a^* = \sqrt{(1-a^*) a^*}.
\end{equation}

Therefore, since $a^* \in (0, \frac{1}{2})$, it follows that $a^* = (5 - \sqrt{5})/10$. It is clear that $a^*$ also maximizes function $f$ and thus, the claim follows if we evaluate $f$ at $a^*$.
\end{proof}

Let us return to the term \eqref{equation:binomial_term}; for an upper bound, it is sufficient to restrict our attention to the case of $i = \Theta(n)$, that is $i = \alpha \cdot n$ for some $\alpha \in (0, 1)$. We will use the following well-known formula:

\begin{lemma}
Stirling's Approximation Formula:

\begin{equation}
    n! \in \Theta\left( \sqrt{2 \pi n} \left( \frac{n}{e} \right)^n \right).
\end{equation}
\end{lemma}

Note that the asymptotic notation is with respect to variable $n$. Therefore, if we apply this formula it follows that

\begin{align}
    \binom{n-i}{i} &= \frac{((1-\alpha)n)!}{(\alpha n)! ((1 - 2\alpha)n)!} \notag \\
                  &\in \Theta \left( \frac{1}{\sqrt{n}} \left(\frac{(1-\alpha)^{1-\alpha}}{\alpha^\alpha (1-2\alpha)^{1-2\alpha}} \right)^n \right).
\end{align}
Thus, if we apply \Cref{lemma:aux} we obtain a bound of $\mathcal{O}(\phi^n/\sqrt{n})$ which is asymptotically tight for $i = \lfloor n (5 - \sqrt{5})/10 \rfloor$.

\begin{figure}[!ht]
    \centering
    \includegraphics[scale=0.48]{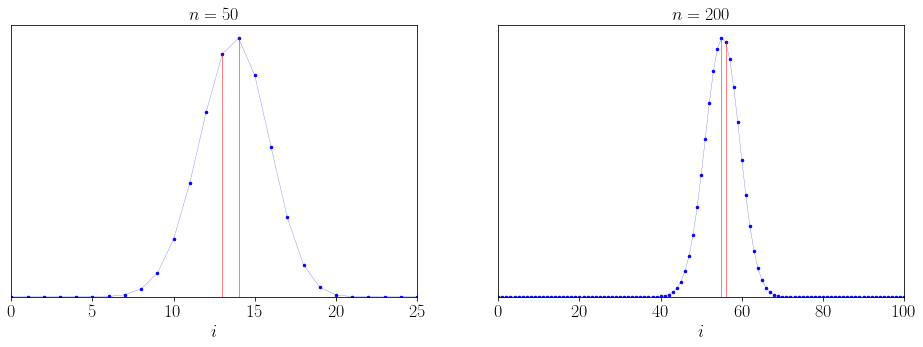}
    \caption{The behavior of the term \eqref{equation:binomial_term} while $i$ progresses from $0$ to $\lfloor n/2 \rfloor$. The red lines are positioned in $i = \lfloor n (5 - \sqrt{5})/10 \rfloor$ and $i = \lceil n (5 - \sqrt{5})/10 \rceil$, and one of them corresponds to the maximum obtained value of the binomial sequence.}
\end{figure}

\end{document}